\documentclass[reqno]{amsart}
\usepackage[utf8x]{inputenc}  %
\usepackage{setspace}
\usepackage[left=2.5cm, right=2.5cm, bottom=2.5cm]{geometry}                 

\usepackage{graphicx} %
\usepackage{pgf,tikz}
\usetikzlibrary{arrows}
\usepackage{amssymb}
\usepackage{amsmath}
\usepackage{pdfsync}
\usepackage{mathrsfs}
\usepackage{hyperref} %
\usepackage{verbatim} %
\usepackage{epstopdf}
\usepackage{bbm} 
\usepackage[colorinlistoftodos,prependcaption,textsize=tiny]{todonotes}
\usepackage{xargs}
\usepackage{lmodern}
\usepackage{mleftright}
\usepackage{enumerate}
\usepackage{caption}
\usepackage{subcaption}
\usepackage[sort&compress,numbers]{natbib} %

\def\R{\mathbb{R}}

\def\Z{\mathbb{Z}}
\def\C{\mathbb{C}}

\def\A{\mathcal{A}}
\def\B{\mathcal{B}}

\def\supp{{\rm supp}}

\def\FF{\mathfrak{F}}

 \newcommand{\du}{\text{\rm d}u}

\renewcommand{\d}{\text{\rm d}}

\newcommand{\eps}{\varepsilon}

\newtheorem{theorem}{Theorem}
\newtheorem{corollary}[theorem]{Corollary}

\newtheorem{proposition}[theorem]{Proposition}
\newtheorem{lemma}[theorem]{Lemma}

\newtheorem*{conjecture}{Conjecture 1}

\newtheoremstyle{newremark}
  {3pt} %
  {3pt} %
  {} %
  {0em} %
  {\sc} %
  {.} %
  {0.5em} %
  {} %
\theoremstyle{newremark}
\newtheorem*{remark}{Remark}

\makeatletter
\DeclareFontFamily{U}{tipa}{}
\DeclareFontShape{U}{tipa}{m}{n}{<->tipa10}{}
\newcommand{\arc@char}{{\usefont{U}{tipa}{m}{n}\symbol{62}}}%

\newcommand{\arc}[1]{\mathpalette\arc@arc{#1}}

\newcommand{\arc@arc}[2]{%
  \sbox0{$\m@th#1#2$}%
  \vbox{
    \hbox{\resizebox{\wd0}{\height}{\arc@char}}
    \nointerlineskip
    \box0
  }%
}
\makeatother

\numberwithin{equation}{section}

\allowdisplaybreaks

\title[Fourier optimization and pair correlation problems]{ Fourier optimization and pair correlation problems }
\author[Das]{Mithun Kumar Das}
\address{ICTP - The Abdus Salam International Centre for Theoretical Physics, Strada Costiera, 11, I - 34151,
Trieste, Italy.}
\address{School of Mathematical Sciences, National Institute of Science Education and Research, A CI of Homi Bhabha National Institute, Jatni, Khurda, 752050, India.}
\email{mdas@ictp.it, mithundas@niser.ac.in}

\author[Ismoilov]{Tolibjon Ismoilov}
\address{SISSA - Scuola Internazionale Superiore di Studi Avanzati, Via Bonomea 265, 34136 Trieste, Italy}
\email{tolibjon.ismoilov@sissa.it}
\author[Ramos]{Antonio Pedro Ramos}
\address{SISSA - Scuola Internazionale Superiore di Studi Avanzati, Via Bonomea 265, 34136 Trieste, Italy}
\email{adeazeve@sissa.it}

\date{\today}                                           
\begin{document}

\begin{abstract} 
{We introduce a generic framework to provide bounds related to the pair correlation of sequences belonging to a wide class. We consider analogues of Montgomery's form factor for zeros of the Riemann zeta function in the case of arbitrary sequences satisfying some basic assumptions, and connect their estimation to two extremal problems in Fourier analysis, which are promptly studied. As applications, we provide average bounds of form factors related to some sequences of number theoretic interest, such as the zeros of primitive elements of the Selberg class, Dedekind zeta functions, and the real and imaginary parts of the Riemann zeta function. In the last case, our results bear an implication to a conjecture of Gonek and Ki (2018), showing it cannot hold in some situations. 
}
\end{abstract}

\maketitle 

\section{Introduction}
The celebrated work of Montgomery \cite{M1973} initiated the study of how non-trivial zeros $\rho = \frac{1}{2} + i \gamma$ of the Riemann zeta function $\zeta(s)$ are distributed along the critical line by studying their pair correlation. His approach consisted of introducing the function $F(\alpha)$, a Fourier transform of the distribution of the gaps between zeros $\gamma - \gamma'$. In \cite{M1973}, Montgomery proposed a conjecture on the limiting distribution of the pair correlation of zeros of zeta, which came to be known as the pair correlation conjecture (PCC). 
Goldston \cite{G} subsequently proved an equivalence between PCC and certain asymptotics for average values of Montgomery's function $F(\alpha)$, motivating the study of average estimates of $F(\alpha)$, which was done, for example, in \cite{CCChiM, CMR, G2, GG}.

The approach of \cite{M1973} was extended to many other settings. There are several instances in the literature where analogues of Montgomery's function were used to study the pair correlation of zeros of $L$-functions, Dirichlet series, and similar objects. One can see \cite{MP1999, FGL2014, BMMRTW2015, KY2017, GK2018, CKL} for a few examples. Naturally, there are corresponding conjectures for their asymptotic behavior. In this work, we will look at the problem of average estimates from a much broader perspective, introducing an axiomatic approach which allows us to bound the long-term averages of analogues of Montgomery's function in a wide variety of settings. We do this by connecting these average bounds to two Fourier optimization problems, which constitute another major object of study of this paper. The investigation of these problems uses techniques from Fourier analysis and the theory of reproducing kernel Hilbert spaces.

As a consequence of the bounds we obtain, we address a conjecture of Gonek and Ki \cite{GK2018} related to the pair correlation of zeros of $\operatorname{Re} \zeta$, $\operatorname{Im} \zeta$, and other similar functions along vertical lines in the left-half of the critical strip. We show that the relevant form factor cannot have the conjectured expression in certain regimes.

\subsubsection{Notation}
Throughout this text, we use the following definition of the Fourier transform for a given function $f\in L^1(\mathbb{R})$:
\begin{equation*}
    \mathscr{F}[f](\xi)= \widehat{f}(\xi):=\int_{\mathbb{R}}e^{-2\pi i \xi x} f(x)\, \d x.
\end{equation*}
We also adopt this normalization for the Fourier transform of a finite Borel measure $\mu$, setting
\begin{equation*}
    \widehat{\mu}(\xi):=\int_{\mathbb{R}}e^{-2\pi i \xi x} \, \d \mu(x).
\end{equation*}
For a real number $x$, we denote $x_+ = \max\{x,0\}$. Finally, the quantity 
\begin{equation*}
    s_0 := \min_{x\in \R} \frac{\sin{x}}{x}=-0.217233\ldots
\end{equation*}
will be useful to state many of our results.

\subsection{Montgomery's pair correlation conjecture}
Assuming the Riemann hypothesis (RH), let $\rho=\frac{1}{2}+i\gamma$ denote a non-trivial zero of $\zeta(s)$.
Montgomery conjectured that, assuming RH, the pair correlation of zeros of zeta has the following limiting distribution: for any fixed $\beta>0$,
\begin{equation}\label{eq:montgomeryconjecture}
   \bigg(\frac{T}{2\pi}\log\,{T}\bigg)^{-1}\sum_{\substack{0<\gamma, \gamma'\leq T \\ 0<\gamma-\gamma'\leq \frac{2\pi\beta}{\log\,{T}}}}1\sim \int_{0}^{\beta}\bigg(  1-\bigg(\frac{\sin\,{\pi u}}{\pi u}\bigg)^2 \bigg)\du, \quad \mbox{as } T\rightarrow \infty.
\end{equation}
This is now known as Montgomery's pair correlation conjecture (PCC). In order to study the quantity on the left-hand side of \eqref{eq:montgomeryconjecture}, Montgomery \cite{M1973} introduced the form factor
\begin{equation}\label{eq:montgomoryffunction}
F(\alpha) = F(\alpha, T):= \bigg(\frac{T}{2\pi}\log\,{T}\bigg)^{-1}\sum_{0<\gamma, \gamma'\leq T}T^{i\alpha (\gamma-\gamma')}w(\gamma\!-\!\gamma'),
\end{equation}
where $\alpha \in \mathbb{R}$, $T\geq 2$, $w(u)=\frac{4}{4+u^2}$, and the sums are taken over ordinates of non-trivial zeros of $\zeta(s)$, counting multiplicity. 

 For any function $r \in L^1(\mathbb{R})$ such that $\widehat{r}\in L^1(\mathbb{R})$, the Fourier inversion formula gives that
\begin{equation*}\label{eq:testingffunctionwithr}
\bigg(\frac{T}{2\pi}\log\,{T}\bigg)^{-1}\sum_{0<\gamma, \gamma'\leq T}r\bigg((\gamma\!-\!\gamma')\frac{\log\, {T}}{2\pi}\bigg)w(\gamma\!-\!\gamma')= \int_{\mathbb{R}}\widehat{r}(\alpha) F(\alpha, T) \, \d\alpha.
\end{equation*}
One can see the connection between \eqref{eq:montgomeryconjecture} and \eqref{eq:montgomoryffunction} from this identity.
Indeed, considering an appropriate sequence of functions $r$ that approximate the characteristic function of the interval $(0, \beta]$ in the above identity, we can obtain the asymptotic of the left-hand side of \eqref{eq:montgomeryconjecture}. Therefore, if we know the asymptotic behavior of the function $F(\alpha, T)$ as $T\to \infty$, we know the asymptotic behavior of the pair correlation of zeros. 

Montgomery \cite{M1973}, then Goldston and Montgomery \cite[Lemma 8]{GM1984}, obtained partial progress on understanding $F(\alpha, T)$ conditionally on RH. They proved that
\begin{equation}\label{eq:asymptmont}
F(\alpha, T) = \left(T^{-2|\alpha|} \log{T} +|\alpha| \right)\left( 1 +O\bigg(\sqrt{\frac{\log \,\log {T}}{\log{T}}}\bigg)\right), \quad \text{ as }\;\; T\to \infty,
\end{equation}
uniformly for $0\leq |\alpha|\leq 1$, improving the original work of Montgomery. Not only this, in \cite{M1973} Montgomery also conjectured that for $|\alpha|\geq 1$ and uniformly for any bounded interval,
\begin{equation}\label{eq:montstrongconjecture}
    F(\alpha, T)\sim 1, \quad \mbox{ as }  T \rightarrow \infty,
\end{equation}
which, together with \eqref{eq:asymptmont}, provides a complete description of the asymptotic behavior of $F(\alpha,T)$. This statement is what became known as the strong pair correlation conjecture, which, by the reasoning above, implies \eqref{eq:montgomeryconjecture}.

In 1988, Goldston \cite{G} proved the equivalence of PCC and the following asymptotic for the averages of $F(\alpha, T)$:
\begin{equation}\label{average of factor form}
    \frac{1}{\ell}\int_{b}^{b+\ell}F(\alpha, T)\d \alpha \sim 1, \mbox{ as } T \rightarrow \infty,
\end{equation}
for any $b\geq 1$ and $\ell>0$.
Recently, Carneiro, Milinovich, and Ramos \cite{CMR} proved, assuming RH, that
\begin{align*}
    0.9303 +o(1) <  \frac{1}{\ell}\int_{b}^{b+\ell}F(\alpha, T)\d \alpha < 1.3208 +o(1), \quad \mbox{ as } T\rightarrow \infty,
\end{align*}
uniformly on $b\geq1$ for the upper bound and for $\ell \geq \ell_0(b)$ for the lower bound, improving the previous results in \cite{CCChiM}. They obtained this result by introducing new averaging mechanisms and making full use of the class of test functions considered in \cite{CE} for the sphere packing problem.

We now proceed to the axiomatic formulation which allows the generalization of the work of \cite{CMR} to address multiple different objects.
\subsection{Axiomatic formulation}
Our main purpose is to obtain a unified approach for various pair correlation functions, capturing the behavior of a wide variety of examples.
Our objects of interest will be sequences  $\Gamma(T) = (\gamma_n(T))$ of real numbers depending on a parameter $T \geq T_0$. Note that the elements of the sequence are not necessarily distinct.
For most of the number-theoretic instances we have in mind, such as when $\Gamma$ consists of the ordinates of non-trivial zeros of a typical $L$-function, the sequence is independent of $T$. We introduce the dependence on the parameter because we are also interested in cases such as the ones Gonek and Ki studied in \cite{GK2018}, where, among other things, they considered the pair correlation of zeros of the real part of $\zeta(s)$ along the lines $\operatorname{Re}(s) = \frac{1}{2} - \frac{c}{\log T}$ for a fixed $c > 0$, which naturally vary with $T$. We discuss this and a few other examples in the dedicated Section \ref{sec:applications} below.

Now, sequences of non-trivial zeros of a typical $L$-function share the behavior that the number of elements with height at most $T > 0$ is asymptotically $\frac{\lambda T}{2 \pi} \log T$ and the average spacing of elements up to that height is $\frac{\lambda}{2 \pi} \log T$ for a given $\lambda > 0$. Motivated by this, we are led to our first axiom.
\smallbreak 
\noindent \textit{Assumption 1} (A1). There is a $\lambda > 0$ such that
\begin{equation}\label{eq:zerocount}
    \liminf_{T\to \infty} \, \frac{\# \{n: \gamma_n(T) \in (0,T] \} }{\frac{\lambda T}{2 \pi} \log T} = 1.
\end{equation}
\smallbreak
In view of this assumption, we associate a form factor to $\Gamma$ in analogy to Montgomery's function \eqref{eq:montgomoryffunction}:
\begin{align}\label{eq:ff}
\FF_\Gamma(\alpha) = \mathfrak{F}_\Gamma(\alpha, T):= \bigg(\frac{\lambda T}{2\pi}\log\,{T}\bigg)^{-1}\sum_{\substack{{\gamma, 
\gamma' \in (0,T] \cap \Gamma(T)  }}} T^{i\lambda \alpha (\gamma-\gamma')}w(\gamma\!-\!\gamma'),
\end{align}
where, as before, $\alpha \in \R$, $T \geq T_0$, $w(u)=\frac{4}{4+u^2}$, and elements in the sum are counted with multiplicity. The choice that the $\gamma$'s lie in the interval $(0,T]$ in \eqref{eq:zerocount} and \eqref{eq:ff} is made for simplicity and is not particularly important. The interval $(0, T]$
can be replaced by, say, $(T, 2T]$ or $[-T/2, T/2]$ with no change in the arguments presented here.
From the definition it also follows that  $\FF_\Gamma(\alpha)$ is real-valued and even, while the identity
\begin{equation*}
    \sum_{\gamma, 
\gamma' \in  (0,T] \cap \Gamma(T)}  T^{i\lambda\alpha(\gamma'-\gamma)}\,w(\gamma'\!-\!\gamma) = 2\pi \int_{-\infty}^{\infty} e^{-4\pi |u|} \Bigg| \sum_{\gamma \in (0,T] \cap \Gamma(T)}   T^{i\lambda\alpha\gamma} e^{2\pi i \gamma u} \Bigg|^2 \du\,,
\end{equation*}
which holds for all $\alpha\in \mathbb{R}$, implies that it is non-negative. 

Below we will give effective bounds for the long-term averages \begin{equation*}
\frac{1}{\ell}\int_{b}^{b+\ell}\FF_\Gamma(\alpha, T)  \, \d\alpha, \, \mbox{ as } T\rightarrow \infty,
\end{equation*}
in this very general situation, assuming  some information is known about the asymptotic behavior of $\FF_\Gamma(\alpha, T)$ as $T\to \infty$. This information will come in the form of a convergence of measures. To properly state this, we introduce a definition. We say a Borel measure $\nu$ is a \emph{limiting measure} for $\FF_\Gamma(\alpha, T)$ in $[-\Delta, \Delta]$ if, for every function $\varphi\in C_c((-\Delta, \Delta))$,
\begin{align}\label{eq:limcondition}
\lim_{T\rightarrow \infty}\int_{\mathbb{R}} \varphi(\alpha) \, \FF_\Gamma(\alpha, T)  \,\d\alpha = \int_{\mathbb{R}} \varphi(\alpha) \,\d\nu(\alpha).
\end{align} 
In the language of measure theory, this means that, as $T \to \infty$, the absolutely continuous measures with Radon--Nikodym derivative $\FF_\Gamma(\alpha, T)$ converge in the weak$^*$ sense to $\nu$ as measures on $[-\Delta, \Delta]$. Such a convergence assumption will constitute the basis of our second axiom on $\Gamma$.

\smallbreak 
\noindent \textit{Assumption 2} (A2). There is a $\Delta > 0$ and a finite Borel measure $\nu$ such that $\nu$ is a limiting measure for $\FF_{\Gamma}(\alpha, T)$ in $[-\Delta, \Delta]$.
\smallbreak
\begin{remark}
    There may be multiple Borel measures on $\R$ for which \eqref{eq:limcondition} holds for all $\varphi \in C_c((-\Delta, \Delta))$. What is relevant here is that, if they exist, they must agree when restricted to the interval $[-\Delta, \Delta]$. 
\end{remark}
As an example, when $\Gamma$ is the set of ordinates of non-trivial zeros of $\zeta(s)$, we have that $\FF_\Gamma(\alpha,T)$ is Montgomery's function $F(\alpha,T)$ and equation \eqref{eq:asymptmont} implies that, under RH, the measure $\nu$ given by
\begin{equation}\label{eq:zetameasure}
   \d\nu(\alpha)=\boldsymbol{\delta}(\alpha)+ |\alpha|\, \d\alpha
\end{equation}
is the limiting measure of $F(\alpha, T)$ in $[-1,1]$, where $\boldsymbol{\delta}$ is the Dirac measure concentrated at the origin. Other examples of form factors and their associated limiting measures can be found in Section \ref{sec:applications}. 

\subsection{Bounds via Fourier optimization}\label{EP}
We obtain lower and upper bounds for the long-term averages of form factors through a Fourier optimization approach. There are two relevant extremal problems which we study, and these will be the key to obtain the desired bounds.
 
The first problem we consider makes use of the class of functions introduced by Cohn and Elkies in  \cite{CE} to tackle sphere packing problems. For a $\Delta>0$, this is the class $\A_\Delta$ of continuous and even functions $g$ such that 
\begin{itemize}
\item[(1)] $g$ and $\widehat{g} \in L^1(\R)$;
\item[(2)] $\widehat{g}(\alpha)\leq 0$ for $\alpha \in \R\setminus [-\Delta, \Delta]$;
\item[(3)] $g(\alpha)\geq 0$ for all $\alpha\in \R$.
\end{itemize}
Now let $\nu$ be a finite Borel measure supported in $[-\Delta, \Delta]$ and define the functional 
\begin{align}\label{functional}
\Phi_\nu(g):= \int_{-\Delta}^{\Delta}\widehat{g}(\alpha)\,\d\nu(\alpha).
\end{align}
Our first  problem is the following.
\smallbreak
\noindent \emph{Extremal Problem 1 }(EP1): Given $\Delta>0$ and a finite Borel measure $\nu$ on $[-\Delta, \Delta]$. Find the quantity
\begin{align}\label{constant}
\mathbf{C}_\nu:=\inf_{\substack{{g\in \A_\Delta} \\ {g(0) > 0}}} \frac{\Phi_\nu(g)}{g(0)}.
\end{align}
If a sequence $\Gamma$ satisfies assumptions (A1) and (A2), we will show that we can provide upper and lower bounds for averages of $\FF(\alpha,T)$ in terms of the optimization constant $\mathbf{C}_\nu$ by applying the averaging mechanism of Carneiro, Milinovich, and Ramos \cite{CMR} in our generic setting, leading to our first main result.

\begin{theorem}\label{thm:boundswithextremalproblem} Let $\Gamma$ be a sequence of real numbers satisfying (A1) and (A2).
Let $\nu$ be the limiting measure for $\FF_\Gamma(\alpha, T)$ in the interval $[-\Delta, \Delta]$. Given $\varepsilon>0$, for any $b\geq 1$ and sufficiently large $\ell$, one has 
  \begin{align*}
   1+s_0(\mathbf{C}_\nu-1)-\varepsilon + o(1)< \frac{1}{\ell}\int_{b}^{b+\ell}\FF_\Gamma(\alpha, T)\, \d\alpha <\mathbf{C}_\nu +\varepsilon + o(1),
  \end{align*}
  as $T \to \infty$, with $\ell \geq \ell_0(\varepsilon)$ for the upper bound and $\ell \geq \ell_0(b,\varepsilon)$ for the lower bound.
\end{theorem} 

\begin{remark}
   By the inequalities above, we have a necessary condition for the limiting measure $\nu$ of $\FF_\Gamma(\alpha, T)$, since the solution of the problem (EP1) must then satisfy $\mathbf{C}_\nu\geq 1$.
\end{remark}

In the next subsections, we discuss how to estimate the optimization constant $\mathbf{C}_\nu$ in order to obtain effective bounds out of Theorem \ref{thm:boundswithextremalproblem}.

We can strengthen the lower bound with the aid of another extremal problem. Given $\beta>0$, we introduce the class $\mathcal R_\beta$, consisting of continuous functions $g$ such that
\begin{enumerate}[(1)]
    \item $g$ and $\widehat g \in L^1(\R)$;
    \item $\widehat{g}(\alpha) \leq \chi_{[-\beta, \beta]} (\alpha)$ for all $\alpha \in \R$; 
    \item $g(\alpha)\geq 0$ for all $\alpha\in \R$.
\end{enumerate}
This definition allows us to introduce the second problem.
\smallbreak
\noindent \emph{Extremal Problem 2 }(EP2): Given $\beta>0$, determine the value of
\begin{equation*}
    \mathbf{D}_\beta := \sup_{g\in  \mathcal R_\beta} g(0) = \sup_{g\in  \mathcal R_\beta} \int_\R \widehat g (\alpha) \, \d \alpha.
\end{equation*}

We solve this problem completely in Section \ref{sec:univlowerbound}, showing that $\mathbf{D}_\beta = \beta$ (see Proposition \ref{prop:ep2solution}). Assuming (A1) and (A2), we obtain a lower bound for (symmetric) averages of $\FF_\Gamma$ in terms of the constant $\mathbf{D}_\beta$. Notably, due to the value  $\mathbf{D}_\beta$, this yields a lower bound of $\frac{1}{2}$ for averages of $\FF_\Gamma$, which is independent of the size of the interval $[-\Delta, \Delta]$.

\begin{theorem}\label{thm:indptlowerbound} Let $\Gamma$ be a sequence of real numbers satisfying (A1) and (A2). Given $\varepsilon>0$, for any $b\geq 1$ and $\ell \geq \ell_0(b,\varepsilon)$, one has 
  \begin{align*}
    \frac{1}{\ell}\int_{b}^{b+\ell}\FF_\Gamma(\alpha, T) \, \d\alpha > \frac{1}{2} - \varepsilon + o(1),
  \end{align*}
  as $T \to \infty$.
\end{theorem} 

\begin{remark}
Our proof of the above theorem shows that one does not need to assume (A2) to obtain an analogous result for averages over symmetric intervals (see Lemma \ref{lem:lowerboundep}). Not only that, one can also see that Theorem \ref{thm:indptlowerbound} is valid under the weaker assumption that there is a $\Delta > 0$ for which $\limsup\limits_{T\to \infty}  \int_{-\Delta}^{\Delta}\FF_\Gamma(\alpha, T) \, \d\alpha$ is finite.
\end{remark}

\subsection{Reproducing kernel Hilbert spaces}
For a given measure $\nu$, finding the exact value of the constant $\mathbf{C}_\nu$ is not an easy task.  We will use reproducing kernel Hilbert space theory to obtain a bound for $\mathbf{C}_\nu$ by working over a subclass of $\mathcal{A}_\Delta$. 
We introduce the relevant background briefly.

A Hilbert space of entire functions $\mathcal{H}$ for which the functionals
\begin{equation*}
    f \mapsto f(w) \in \C
\end{equation*}
are continuous for all $w \in \C$ is said to be a reproducing kernel Hilbert space because, as a consequence of the Riesz representation theorem, there exists a function $K: \C^2 \to \C$ such that $K(w, \cdot) \in \mathcal H$ and
\begin{equation*}
    f(w) = \langle f, K(w, \cdot) \rangle,
\end{equation*}
for all $f \in \mathcal H$ and $w \in \C$. This $K$ is called the reproducing kernel of $\mathcal H$.

A classical example of such a space is the Paley--Wiener space $PW(\pi \Delta)$ for a $\Delta > 0$. We recall that a function $f: \C \to \C$ is said to be of exponential type at most $\pi \Delta$ if for a given $\varepsilon > 0$ there exists a constant $C_\varepsilon$ such that
\begin{equation*}
    |f(z)| \leq C_\varepsilon e^{ (\pi \Delta + \varepsilon)|z|} 
\end{equation*}
holds for all $z \in \C$. The space $PW(\pi \Delta)$ consists of all entire functions $f: \C \to \C$ of exponential type at most $\pi \Delta$ such that $f|_{\R} \in L^2(\R)$. This forms a reproducing kernel Hilbert space with the usual $L^2(\R)$-inner product. We note that the Paley--Wiener theorem establishes that the set of functions in $L^2(\R)$ which can be extended to a function in $PW(\pi \Delta)$ coincides with the $L^2(\R)$-functions whose Fourier transforms are supported in $[-\frac{\Delta}{2}, \frac{\Delta}{2}]$.

Here we will be interested in weighted versions of $PW(\pi \Delta)$. Given a finite Borel measure $\nu$ supported in $[-\Delta, \Delta]$, we consider the set $\mathcal H_{\nu}$ consisting of entire functions of exponential type at most $\pi \Delta$ satisfying
\begin{equation*}
    \| f\|_{\nu}^2 := \int_\R |f(x)|^2 \, \widehat{\nu}(x) \, \d x < \infty.
\end{equation*}
It is not clear a priori that the functional above is a norm and that $\mathcal{H}_\nu$ is even a vector space. One must have specific information about the measure $\nu$ to be able to say anything more concrete. Therefore, we restrict our attention to a class of measures which addresses many interesting cases relevant from a number theory standpoint, namely, the measures $\nu$ supported in $[-\Delta, \Delta]$ given by 
\begin{equation}\label{eq:classofmeasures}
     \d\nu(\alpha)= c_1\boldsymbol{\delta}(\alpha)+c_2|\alpha|e^{-c_3|\alpha|}\d\alpha,
\end{equation} 
for $\alpha \in [-\Delta, \Delta]$, where $c_1 > 0$ and $c_2, c_3 \geq 0$ are constants.
In these cases, we determine that the sets $\mathcal H_{\nu}$ satisfy the following.
\begin{theorem}\label{thm:rkhs}
     Let $\nu$ be a measure in $[-\Delta, \Delta]$ given by $\d\nu(\alpha)= c_1\boldsymbol{\delta}(\alpha)+c_2|\alpha|e^{-c_3|\alpha|}\d\alpha$. For $\frac{c_2}{c_1}\Delta^2\leq 5/3$ and $c_3\geq 0$, we have that $\| \cdot\|_{\nu}$ is a norm which makes the set $\mathcal H_{\nu}$ a reproducing kernel Hilbert space. Moreover, we have that $\mathcal H_{\nu} = PW(\pi \Delta)$ as sets, and we have the equivalence of norms
     \begin{equation}\label{eq:equivalentnorm}
        a \| f\|_{L^2(\R)} \leq \| f\|_{\nu} \leq b \| f\|_{L^2(\R)} ,
     \end{equation}
     for constants $a,b > 0$ depending only on the parameters $c_1, c_2$, and $\Delta$.
\end{theorem}
To establish this theorem, we find bounds from above and below to the function $\widehat{\nu}(x)$, which introduces the need for the constraint that $\frac{c_2}{c_1}\Delta^2\leq 5/3$. This is a sufficient condition on the parameters to obtain this result, but it is (almost) the best that can be done using this method. We can slightly extend the range by increasing the technicality of the proof, but opt for the value $5/3$ for simplicity and clarity of exposition. 
 
By working over the subclass of $\mathcal{A}_\Delta$ consisting of functions $f$ with $\supp \widehat{f} \subseteq [-\Delta, \Delta]$, we can take advantage of the reproducing kernel structure of $\mathcal{H}_\nu$ to obtain an upper bound for $\mathbf{C}_\nu$. This restricted version of (EP1) was first solved by Montgomery and Taylor \cite{M2} in the case of the measure given by \eqref{eq:zetameasure}, i.e., the limiting measure associated with zeros of the Riemann zeta function. They used a variational method, while Carnerio, Chandee, Milinovich, and Littmann \cite{CCLM} used reproducing kernel Hilbert space theory to arrive at the same result. We adapt the proof of \cite[Corollary 14]{CCLM} to our context, establishing the lemma below.
\begin{lemma}\label{lem:rkbound}
     Let $\nu$ be a finite Borel measure supported in $[-\Delta, \Delta]$ such that the inequalities \eqref{eq:equivalentnorm} hold. Then $(\mathcal H_{ \nu}, \| \cdot \|_{\nu})$ is a reproducing kernel Hilbert space such that $\mathcal{H}_\nu = PW(\pi \Delta)$ as sets and
    \begin{equation*}
         \mathbf{C}_\nu \leq \frac{1}{K_\nu(0,0)},
    \end{equation*}
    where $K_\nu$ is its reproducing kernel.
\end{lemma}
This is the link between explicit bounds and reproducing kernels which underpins our strategy. We include a short proof of this lemma in Section \ref{sec:rkhs}.

\subsection{Explicit bounds for averages}
As we have just seen, one can produce an effective bound for $\mathbf{C}_\nu$ in the case that $(\mathcal H_{ \nu}, \| \cdot \|_{\nu})$ is a reproducing kernel Hilbert space if one knows the value of the reproducing kernel $K_\nu(w,z)$ at $(w,z) = (0,0)$. Accordingly, the next results concern the value of $K_\nu(0,0)$, separated in two cases because of the significant increase in computational difficulty when $c_3 \neq 0$.
\begin{theorem}\label{thm:k00firstcase}
    Let $\Delta>0,\, c_1> 0,\, c_2\geq 0$ be such that $\frac{c_2}{c_1}\Delta^2 \leq 5/3$. For the measure $\nu$ supported in $[-\Delta, \Delta]$ given by $\d\nu(\alpha)= c_1\boldsymbol{\delta}(\alpha) + c_2|\alpha|\,\d\alpha$, we have
    \begin{equation*}
        K_\nu(0,0) =  \sqrt{ \frac{2 }{c_1 c_2}} \frac{\sin{\left(\sqrt{\frac{c_2}{2c_1}} \Delta \right)}}{\cos{\left( \sqrt{\frac{c_2}{2c_1}} \Delta \right)}+\sqrt{\frac{c_2}{2c_1}} \Delta \sin{\left(\sqrt{\frac{c_2}{2c_1}} \Delta \right)}}.
    \end{equation*}
    \end{theorem}

    \noindent Now, to state the result when $c_3 \neq 0$, we define the auxiliary functions of $\eta \in \C$
    \begin{align}
        A(\eta)&:=1+\frac{c_2}{c_1}\int_{-\Delta/2}^{\Delta/2}\cosh{(\eta\alpha)}\lvert \alpha\rvert e^{-c_3\lvert\alpha\rvert}\,\d\alpha;\label{eq:auxiliaryA}\\
        B(\eta)&:=\eta^2+2\frac{c_2}{c_1}-c_3^2+2\frac{c_2 c_3}{c_1}\int_{-\Delta/2}^{\Delta/2}\cosh{(\eta\alpha)}e^{-c_3\lvert\alpha\rvert}\,\d\alpha.\label{eq:auxiliaryB}
    \end{align}
    \begin{theorem}\label{thm:k00secondcase}
        Let $\Delta, c_1, c_2,\text{ and } c_3$ be positive numbers such that $\frac{c_2}{c_1}\Delta^2\leq 5/3$. For the measure $\nu$ supported in $[-\Delta, \Delta]$ given by $\d\nu(\alpha)=c_1\boldsymbol{\delta}(\alpha)+c_2\lvert \alpha\rvert e^{-c_3\lvert \alpha\rvert}\d\alpha$, we have 
        \begin{align}
                \nonumber K_\nu(0, 0)=&\frac{2}{A(\eta_1)B(\eta_2)-B(\eta_1)A(\eta_2)}\left(\frac{1}{c_1}-\frac{A(0)c_3^2}{2c_2+c_3^2 c_1}\right)\left(\frac{\sinh{\frac{\Delta \eta_1}{2}}}{\eta_1}B(\eta_2)-\frac{\sinh{\frac{\Delta \eta_2}{2}}}{\eta_2}B(\eta_1)\right)\\
                \nonumber &+\frac{2}{A(\eta_1)B(\eta_2)-B(\eta_1)A(\eta_2)}\left(\frac{c_3^2}{c_1}+\frac{B(0)c_3^2}{2c_2+c_3^2 c_1}\right)\left(\frac{\sinh{\frac{\Delta \eta_1}{2}}}{\eta_1}A(\eta_2)-\frac{\sinh{\frac{\Delta \eta_2}{2}}}{\eta_2}A(\eta_1)\right)\\
                &\label{eq:k00exp}+\frac{\Delta c_3^2}{2c_2+c_3^2 c_1},
            \end{align}
    where the numbers $\eta_1$ and $\eta_2$ are roots of the equation 
    \begin{equation*}
        c_1\eta^4 +2\left({c_2}- c_1 c_3^2\right)\eta^2+c_3^2\left(2{c_2} + c_1c_3^2\right)=0
    \end{equation*}
    such that $\eta_1+\eta_2\neq 0$. The expression \eqref{eq:k00exp} is interpreted as a limit whenever $\eta_1=\eta_2$.
    \end{theorem}
    \begin{remark}
        When $c_3$ is sufficiently large, we have the following asymptotic description for $ \frac{1}{K_\nu(0,0)}$:
        \begin{equation*}
                 \frac{1}{K_\nu(0,0)}= \frac{c_1}{\Delta} + O\left(\frac{c_2 \Delta}{c_3}\right).
        \end{equation*}
    \end{remark}

In Section \ref{sec:computing-rks}, we actually compute the explicit expression of the full reproducing kernel $K_\nu(w,z)$ when $c_3 = 0$, and of $K_\nu(0,z)$ when $c_3 > 0$. These are the contents of Theorems \ref{thm:fullexpressionkwz} and \ref{thm:reprokernelkzero}, respectively. The reproducing kernel is a solution of a certain integral equation depending on $\nu$.  We obtain such a solution for the class of measures given by \eqref{eq:classofmeasures} by constructing and solving a set of suitable ordinary differential equations. This constitutes Lemmas \ref{lem:functionalodesimple} and \ref{lem:functionalodeexp}. 

Finally, we can combine all the previous results to obtain an estimate for the long-term averages of the function $\FF_\Gamma(\alpha, T)$ associated to a sequence $\Gamma$ satisfying (A1) and (A2) with the limiting measure given by $\d\nu(\alpha)= c_1\boldsymbol{\delta}(\alpha)+c_2|\alpha|e^{-c_3|\alpha|}\d\alpha$.
\begin{corollary}\label{thm:explicitboundswithreprokernel}
Let $\Gamma$ be a sequence of real numbers satisfying (A1) and (A2).
Assume $\nu$, the limiting measure for $\FF_\Gamma(\alpha, T)$ in the interval $[-\Delta, \Delta]$, is given by $\d\nu(\alpha)= c_1\boldsymbol{\delta}(\alpha)+c_2|\alpha|e^{-c_3|\alpha|}\d\alpha$. If $\frac{c_2}{c_1}\Delta^2\leq 5/3$ and $c_3\geq 0$, then, for all $\varepsilon>0$, $b\geq 1$ and for sufficiently large $\ell$, one has
\begin{align*}
  \bigg(\frac{1}{2}+s_0\bigg(\frac{1}{K_\nu(0,0)}-1\bigg)\!\!\bigg)_{\!\!+} + \frac{1}{2}-\varepsilon + o(1) < \frac{1}{\ell}\int_{b}^{b+\ell}\FF_\Gamma(\alpha, T)\d\alpha < \frac{1}{K_\nu(0,0)} +\varepsilon + o(1),
\end{align*}
as $T \to \infty$, with $\ell \geq \ell_0(\varepsilon)$ for the upper bound and $\ell \geq \ell_0(b,\varepsilon)$ for the lower bound, where $K_\nu(0,0)$ is given by Theorem \ref{thm:k00firstcase} and Theorem \ref{thm:k00secondcase}.
\end{corollary} 

\subsection{Applications}\label{sec:applications} Letting $\Gamma(T)$ be a sequence coming from zeros of a Dirichlet series or similar object, possibly varying with $T$, we can obtain upper and lower bounds for averages of $\FF_\Gamma(\alpha, T)$ as a consequence of our results. We treat three main classes of examples:
\begin{enumerate}
    \item Primitive elements of the Selberg class;
    \item Dedekind zeta functions of number fields;
    \item Real and imaginary parts of the Riemann zeta function (and suitable generalizations of these).
\end{enumerate}

\subsubsection{Primitive $L$-functions in $\mathcal{S}$ of degree $m$} The Selberg class $\mathcal{S}$, introduced by Selberg in 1989 \cite{S1989}, is the set of meromorphic functions $L(s)$ satisfying five fundamental axioms: Dirichlet series representation, analytic continuation, functional equation, Ramanujan hypothesis, and Euler product. Associated with each $L$-function in $\mathcal{S}$ is its degree $d_L$, which is determined by the parameters in its functional equation. Although not assumed in the axioms, it is conjectured that the degree of any element in the Selberg class is a non-negative integer, which is supported by all known examples. Some examples of elements in $\mathcal{S}$ include:
\begin{itemize}
        \item The Riemann zeta function $\zeta(s)$ (degree 1);
        \item Dirichlet $L$-functions $L(s, \chi)$ and their imaginary shifts $L(s+i\theta, \chi)$, where $\chi$ is a primitive character (degree 1);
        \item Dedekind zeta functions $\zeta_K(s)$ of number fields $K$ (degree $[K:\mathbb{Q}]$);
        \item $L$-functions associated to irreducible cuspidal automorphic representations of $GL_m$ over $\mathbb{Q}$ (degree $m$).
\end{itemize}

Murty and Perelli \cite{MP1999} studied the pair correlation of zeros of primitive functions in $\mathcal{S}$, which we recall are elements of this class that cannot be further factored into non-trivial elements of $\mathcal{S}$. Here, we will be considering sequences $\Gamma$ consisting of ordinates of non-trivial zeros of primitive elements $L$ of $\mathcal{S}$, i.e., $\Gamma = \{\gamma: \rho = \beta + i \gamma ,\, L(\rho) = 0,\, 0 \leq \beta \leq 1\}$. An application of the argument principle shows that for such a typical element $L \in \mathcal{S}$, the zero counting function is given by
\begin{equation*}
    N_L(T) = \#\{\rho = \beta + i \gamma : L(\rho) = 0, 0 \leq \beta \leq 1, -T/2 \leq \gamma \leq T/2  \} = \frac{d_L T}{2 \pi} \log T + O(T),
\end{equation*}
proving $\Gamma$ satisfies axiom (A1) with the interval $(0,T]$ replaced by $[-T/2, T/2]$ in \eqref{eq:zerocount}, and with $\lambda = d_L$. Accordingly, we define the form factor $\FF_\Gamma(\alpha,T)$ with a slight change in the definition in \eqref{eq:ff}, replacing the intervals.
As we have remarked previously, this change does not create any issues and all of our results still apply. We make this choice for easy comparison with the set-up of \cite{MP1999}, where the authors studied the asymptotic behavior of form factors defined similarly to above under an assumption on the Dirichlet series coefficients of $\log L(s).$ 

Recall that the Euler product axiom in the definition of the Selberg class states that, for any $L\in \mathcal{S}$,
\begin{align*}
\log{L(s)} = \sum_{n=1}^\infty \frac{b_L(n)}{n^s},
\end{align*}
where $b_L(n)=0$ unless $n=p^k$, $k\geq 1$ and $b_L(n)\leq n^\theta$ for some $\theta < 1/2$. 
Murty and Perelli proposed the following hypothesis about the coefficients $b_L (n)$.

\smallskip
\noindent \textbf{Hypothesis MP.} Let $v_L(n) := b_L(n) \log{n}$. We have
\begin{align*}
\sum_{n\leq x}v_L(n)\overline{v_L(n)}= (1+o(1))x\log{x}, \mbox{ as } x\rightarrow \infty.
\end{align*}
They showed that this hypothesis together with GRH implies
\begin{align*}
\FF_\Gamma(\alpha, T)\sim 
d_LT^{-2\lvert \alpha\rvert d_L} \log{T} + \lvert\alpha\rvert, &  \mbox{ as } T\rightarrow \infty,  
\end{align*}
uniformly for $0\leq \lvert \alpha\rvert\leq \frac{1}{d_L}(1-\delta)$, for any $\delta>0$. Therefore, $\Gamma$ satisfies (A2) with the limiting measure $\nu$ of $\FF_\Gamma(\alpha, T)$ in the interval $\left[-\frac{1}{d_L}, \frac{1}{d_L}\right]$ given by 
\begin{equation*}\d\nu(\alpha)=\boldsymbol{\delta}(\alpha)+\lvert \alpha\rvert \d\alpha.
\end{equation*}
As consequence of the above, by Corollary \ref{thm:explicitboundswithreprokernel} we obtain the following.
\begin{corollary}\label{coro2}
Assume GRH and Hypothesis MP, if $\Gamma$ is the sequence of ordinates of the non-trivial zeros of a primitive element of $\mathcal{S}$ of degree $m$, we have that, for all $\varepsilon > 0$ and $b\geq 1$,
\begin{equation*}
     \bigg( \frac{1}{2} +s_0\bigg(\frac{1}{\sqrt{2}}\cot{\frac{1}{\sqrt{2}m}} - \frac{2m -1}{2m}\bigg)\!\!\bigg)_{\!\!+} + \frac{1}{2} -\varepsilon + o(1) < \frac{1}{\ell}\int_{b}^{b+\ell}\FF_\Gamma(\alpha, T)\d\alpha < \frac{1}{\sqrt{2}}\cot{\frac{1}{\sqrt{2}m}} +\frac{1}{2m} +\varepsilon + o(1),
\end{equation*}
as $T\to \infty$, with $\ell\geq \ell_0(\varepsilon)$ for the upper bound and $\ell\geq \ell_0(b, \varepsilon)$ for the lower bound. 
\end{corollary}

\subsubsection{Dedekind zeta functions of number fields}
The Dedekind zeta function $\zeta_K(s)$ over a number field $K/\mathbb{Q}$ of degree $n$ with ring of integers $\mathcal{O}_K$ is defined as
\begin{align*}
    \zeta_K(s)= \sum_{I \subseteq \mathcal{O}_K}\frac{1}{N(I)^s},
\end{align*}
where the sum is over the ideals $I$ in $\mathcal{O}_K$  and $N(I)$ is the norm of $I$. It is known that for an abelian extension $K/\mathbb{Q}$, 
\begin{align*}
    \zeta_K(s) =\zeta(s) \prod_{j=1}^{n-1}L(s, \chi_j),
\end{align*}
for some Dirichlet characters $\chi_j$, $j = 1, \ldots n$, is a non-primitive $L$-function of degree $n$. In a recent preprint \cite{DRTW}, de Laat, Rolen, Tripp, and Wagner studied the pair correlation of zeros of $\zeta_K(s)$ under the Grand Riemann Hypothesis (GRH). If $\Gamma$ consists of the ordinates of the zeros of $\zeta_K$ in the critical strip, standard results confirm it satisfies assumption (A1). In \cite{DRTW}, the authors studied the form factor $\FF_\Gamma$ (defined as in \eqref{eq:ff}) and established that 
\begin{equation*}
\FF_\Gamma(\alpha, T)\sim 
nT^{-2n |\alpha|} \log{T} + n|\alpha|,
\end{equation*}
as $T\to \infty$, uniformly for $\lvert\alpha\rvert \leq\frac{1}{n}-\delta$, with $\delta>0$.
This gives that axiom (A2) also holds for $\Gamma$, with limiting measure $\nu$ in the interval $[-\frac{1}{n}, \frac{1}{n}]$ given by
\begin{equation*}
    \d\nu(\alpha) = \boldsymbol{\delta}(\alpha) + n |\alpha| \, \d\alpha,
\end{equation*}
whence we get the following corollary of Theorem \ref{thm:explicitboundswithreprokernel}.
\begin{corollary}\label{coro3} Assuming GRH, if $\Gamma$ is the sequence of zeros of the Dedekind zeta function $\zeta_K(s)$ of a degree $n$ abelian extension $K/\mathbb{Q}$, we have that, for all $\varepsilon > 0$ and $b\geq 1$,
  \begin{align*}
  \bigg(\frac{1}{2}+s_0\bigg(\sqrt{\frac{n}{2}}\cot{\frac{1}{\sqrt{2n}}}-\frac{1}{2}\bigg)\!\!\bigg)_{\!\!\!+} + \frac{1}{2} -\varepsilon + o(1) < \frac{1}{\ell}\int_{b}^{b+\ell}\FF_\Gamma(\alpha, T)\d\alpha < \sqrt{\frac{n}{2}}\cot{\frac{1}{\sqrt{2n}}} +\frac{1}{2}+\varepsilon+o(1),
  \end{align*}
   as $T\to \infty$, with $\ell\geq \ell_0(\varepsilon)$ for the upper bound and $\ell\geq \ell_0(b, \varepsilon)$ for the lower bound. 
\end{corollary}
\subsubsection{The real and imaginary parts of the Riemann zeta function} Another application of our results is about the behavior of the form factor for the zeros of the real and imaginary parts of the Riemann zeta function and similar functions. 
For $0<a\leq \frac{1}{2}$, consider the following generalization of $\operatorname{Re} \zeta$ and $\operatorname{Im} \zeta$ investigated in \cite{GK2018, Ga, Ki}: 
\begin{align*}
f_a(w, \theta) = \zeta(a+w) +e^{i\theta}\zeta(a-w), \mbox{ where } \theta\in [0, 2\pi).
\end{align*}
Let $\rho_a= \beta_a + i\gamma_a$ denote the typical zero of $f_a(w,\theta)$. 
We highlight that the zeros of $f_a(w,0)$ and $f_a(w,\pi)$ coincide respectively with those of $\operatorname{Re} \zeta (a + w)$ and  $\operatorname{Im} \zeta (a + w)$ on the line $\operatorname{Re}w = 0$.
Ki \cite{Ki} also showed that, under RH, all but a finite number of non-real zeros of $f_a(w, \theta)$ lie on $\operatorname{Re}{w}=0$ for any $0<a\leq \frac{1}{2}$. Thus, there exists a real number $T_a$ such that $\beta_a=0$ when $\gamma_a > T_a$. As pointed out in \cite{GK2018}, an inspection of the proof of Ki reveals this lower bound can be taken uniform, in the sense that there exists a $T_0$ which works for all $a \in (0,\frac{1}{2}]$ given a fixed $\theta \in [0,2\pi)$.

Assuming RH, Gonek and Ki \cite{GK2018} proved that the zeros of $f_a(w, \theta)$ are simple after a sufficiently large height for fixed $0<a< \frac{1}{2}$ and $\theta \in [0,2\pi)$, and that, under the same conditions, the consecutive zeros on the line $\operatorname{Re}{w}=0$ are regularly spaced.  However, when $a$ is close to $\frac{1}{2}$, the spacing between consecutive zeros is quite unpredictable. Hence, we focus on the case $a=\frac{1}{2}-\frac{c}{\log{T}}$, where $c>0$, and let $\Gamma(T)$ be the sequence of ordinates $(\gamma_a)$ for $T > T_0$ (here $\theta$ is still fixed). We make a slight change in the setup to match the definitions of \cite{GK2018}: we work with elements $\gamma, \gamma'\in \Gamma(T)$ from the interval $[T, 2T]$ instead of $[0,T]$.  As mentioned previously, this implicates no change in the arguments presented here. Now, Garaev \cite{Ga} and Ki \cite{Ki} have established that
\begin{equation*}
    N_a(T) = \#\{\rho_a: T \leq \gamma_a\leq 2T\} = \frac{T}{\pi}\log T + O(T),
\end{equation*}
which means $\Gamma(T)$ satisfies assumption (A1) with $\lambda = 2$ and replacing the interval $(0,T]$ by $[T,2T]$ in \eqref{eq:zerocount}. We thus also define $\FF_\Gamma(\alpha, T)$ taking elements $\gamma, \gamma'$ from the interval $[T, 2T]$ in  \eqref{eq:ff}.

For the form factor associated with the zeros of $f_a(w, \theta)$, Gonek and Ki \cite[Theorem 3]{GK2018} established the following \footnote{We remark that in \cite{GK2018}, the definition of the form factor has a scaling different from ours, so we change it accordingly.}. 
Let $0<\delta < 1/2$. Under RH, for $T\geq T_0$ and $\delta \leq a\leq 1/2$,
\begin{align*}
\FF_\Gamma(\alpha, T) \sim 2 T^{-4|\alpha|} \log{T} + \frac{2|\alpha|}{(a+1/2)(3-2a)}T^{(4a-2)|\alpha|},  
\end{align*}
as $T\to \infty$, uniformly for $0\leq |\alpha| \leq \frac{1}{2}-\delta$ and $\delta \leq a\leq 1/2$. 
From this, we obtain that $\Gamma(T)$ satisfies (A2) with limiting measure $\nu$ in the interval $[-\frac{1}{2}, \frac{1}{2}]$ given by
\begin{align*}
   \d\nu(\alpha)= \boldsymbol{\delta}(\alpha) + |\alpha|e^{-4c|\alpha|}\d\alpha.
\end{align*}
Hence we get the following result by applying Corollary \ref{thm:explicitboundswithreprokernel}.
\begin{corollary}\label{coro:reimzeta}
Assuming RH, if $\theta \in [0,2\pi)$, $a=\frac{1}{2}-\frac{c}{\log{T}}$, $c\geq0$, and $\Gamma(T)$ consists of the zeros of $f_a(w,\theta)$ along the line $\operatorname{Re}w = 0$, we have that, for all $\varepsilon > 0$ and $b\geq 1$,
\begin{align*}
        1+s_0\bigg(\frac{1}{K_\nu(0,0)} -1\bigg) -\varepsilon + o(1) < \frac{1}{\ell}\int_{b}^{b+\ell}\FF_\Gamma(\alpha, T)\d\alpha < \frac{1}{K_\nu(0,0)} +\varepsilon + o(1),
\end{align*}
as $T \to \infty$, with $\ell\geq \ell_0(\varepsilon)$ for the upper bound and $\ell\geq \ell_0(b, \varepsilon)$ for the lower bound, where $K_\nu(0, 0)$ is given by Theorem \ref{thm:k00firstcase} and Theorem \ref{thm:k00secondcase} with $\Delta = \frac{1}{2}, c_1 = 1, c_2 =1$, and $c_3 = 4c$.

In particular, if $\Gamma$ is either the sequence of zeros of $\operatorname{Re} \zeta$ or that of $\operatorname{Im} \zeta$ along the $\frac{1}{2}$-line, we have that, for all $\varepsilon > 0$ and $b\geq 1$,
   \begin{align*}
       0.7467\cdots -\varepsilon + o(1)< \frac{1}{\ell}\int_{b}^{b+\ell}\FF_\Gamma(\alpha, T)\, \d\alpha <2.1659\cdots +\varepsilon +o(1),
       \end{align*}
 as $T\to \infty$, with $\ell\geq \ell_0(\varepsilon)$ for the upper bound and $\ell\geq \ell_0(b, \varepsilon)$ for the lower bound.
\end{corollary}
\begin{figure}[t]
\centering
\includegraphics[width=0.5\linewidth]{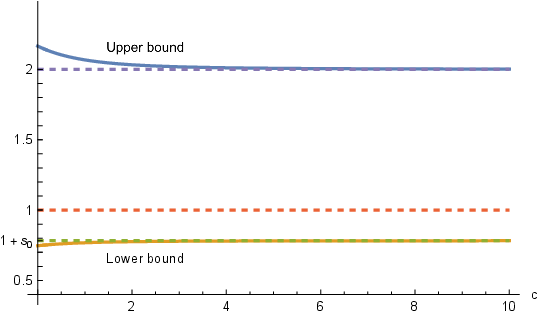}
\caption{Upper and lower bounds for $ \frac{1}{\ell}\int_{b}^{b+\ell}\FF_\Gamma(\alpha, T)\,\d\alpha $ as a function of $c$, where $c_1=1$, $c_2=1$, $c_3 = 4c$, and $\Delta  = 0.5$.}
\label{fig:boundsforF}
\end{figure}

In Figure \ref{fig:boundsforF}, we illustrate the upper and lower bounds given in the above corollary.

One implication of the average estimates of Corollary \ref{coro:reimzeta} has to do with \cite[Conjecture 2]{GK2018}, in which Gonek and Ki claim the following.
\begin{conjecture}
    Suppose $\theta \in [0,2\pi)$ is fixed, $a=1/2-c/\log T$ with $0<c=o(\log T)$, and that $A>0$ is arbitrarily large but fixed. Then, if $\Gamma(T)$ consists of the zeros of $f_a(w,\theta)$ along $\operatorname{Re}w = 0$,
    \begin{equation*}
        \FF_\Gamma(\alpha, T)=(2+o(1))T^{-4\lvert \alpha\rvert}\log T+\frac{\min (2 \lvert \alpha\rvert, 1)}{(a+1/2)(3-2a)}T^{(4a-2)\lvert \alpha\rvert}+o(1)
    \end{equation*}
    uniformly for $0\leq \lvert \alpha\rvert\leq A$.
\end{conjecture}
According to this conjecture we should have, for $c > 0$ fixed, that 
\begin{align*}
    \lim_{T\to \infty}\frac{1}{\ell}&\int_{b}^{b+\ell}\FF_\Gamma(\alpha, T)\,d\alpha \\
    &= \frac{2}{\ell}\lim_{T\to\infty}\int_{b}^{b+\ell}T^{-4\alpha}\log T\,d\alpha
    +\frac{1}{\ell}\lim_{T\to\infty}\frac{1}{\left(1-\frac{c}{\log T}\right)\left(2-\frac{2c}{\log T}\right)}\int_{b}^{b+\ell}T^{-4c\alpha/\log T}\,d\alpha\\
    &=\frac{1}{2\ell}\frac{e^{-4c b}(1-e^{-4c \ell})}{4c}.
\end{align*}
for any $b>\frac{1}{2}$ and $\ell>0$. As either $b \to \infty$ or $\ell\to \infty$,
\begin{equation*}
\frac{1}{2\ell}\frac{e^{-4c b}(1-e^{-4c \ell})}{4c}\to 0,
\end{equation*}
so that for sufficiently large $b$ or $\ell$ we should be seeing the averages of the function $\FF_\Gamma(\alpha, T)$ converging to 0. By the lower bounds of Corollary \ref{coro:reimzeta}, this cannot happen. 

Therefore, we conclude Conjecture 1 does not hold in this generality.
It is an interesting question whether the conjecture is still true when $c = o(1)$.

\section{Fourier Optimization Bounds}

\subsection{Proof of Theorem \ref{thm:boundswithextremalproblem}} Fix $g\in \A_\Delta$ with $g(0)>0$ and let $\phi$ be a  Schwartz function satisfying $\phi, \widehat{\phi} \geq 0$ and $\|\phi\|_1=1$ with $\supp{(\widehat{\phi})}\subset [-1,1]$. Define $\phi^\varepsilon(x)=\frac{1}{\varepsilon}\phi(x/\varepsilon)$ and set $g_\varepsilon:= g*\phi^\varepsilon$. We find that $g_\varepsilon$ occurs in $\A_\Delta^{BL}$, the collection of functions $f\in \A_\Delta$ such that $\operatorname{supp}{(\widehat{f})}$ is compact, and $\lim_{\varepsilon \to 0}\frac{\Phi_\nu(g_\varepsilon)}{g_\varepsilon(0)} = \frac{\Phi_\nu(g)}{g(0)}$. Thus,  given a $\delta>0$ , there is a $g \in \A_\Delta$ and a suitably small $\varepsilon$ for which $\frac{\Phi_\nu(g_\varepsilon)}{g_\varepsilon(0)} \leq \mathbf{C}_\nu +\delta$.

Observing this fact, for a given $\varepsilon>0$ there is a function $h\in \A_\Delta^{BL}\subset \A_\Delta$, that satisfies the inequality
\begin{align*}
    \Phi_\nu(h)\leq \mathbf{C}_\nu +\frac{\varepsilon}{3}.
\end{align*}
Without loss of generality, we normalize $h$ to have $h(0)=1$. Now, pick $g \in \A_{(1 - \delta)\Delta}^{BL}$ defined on the Fourier transform side by $\widehat{g}(\alpha) = \frac{1}{1-\delta}\widehat{h}\left(\frac{1}{1-\delta} \alpha\right)$ for a $0 < \delta < \frac{1}{2} $ small enough such that
\begin{equation*}
    \int_{-(1-\delta)\Delta}^{(1-\delta)\Delta} \widehat{g} (\alpha) \, \d \nu (\alpha) \leq \Phi_\nu(h) + \frac{\varepsilon}{3}.
\end{equation*}
and  $\operatorname{supp}{(\widehat{g})} \subset [-M, M]$, with $M$ a parameter uniform in $\delta$.

The core of the proof is an inequality regarding a convenient approximation of the characteristic functions of intervals (established by Carneiro, Milinovich, and Ramos in \cite[eq: (2.6)]{CMR} and reproduced here as \eqref{inequality 1}). For $L>M$, define the function
\begin{align*}
\mathcal{G}_L(x) = (\widehat{g}*\chi_{[-L, L]})(x)=\int_{x-L}^{x+L}\widehat{g}(t) \d t.
\end{align*}

We first address the upper bound. Setting $L=M+\frac{\ell}{2}$, it follows that  
\begin{align}\label{inequality 1}
    \mathcal{G}_L\bigg(x-b-\frac{\ell}{2}\bigg) \geq \chi_{[b, b+\ell]}(x) -\|\widehat{h}\|_1\cdot \chi_{b+\frac{\ell}{2}+I_L}(x),
\end{align}
where $I_L:= [-L-M, -L+M]\bigcup [L-M, L+M]$. The above inequality implies that 
\begin{align}\label{upper bound estimate}
    \int_{b}^{b+\ell}\FF_\Gamma(\alpha, T) \, \d \alpha \leq \int_{-\infty}^\infty \mathcal{G}_L\bigg(\alpha -b-\frac{\ell}{2}\bigg) \, \FF_\Gamma(\alpha, T)\,  \d \alpha +  \|\widehat{h}\|_1\int_{b + \frac{\ell}{2} + I_L}\FF_\Gamma(\alpha, T)\, \d \alpha.
\end{align}
By the definiton of $\A_\Delta$ and assumption (A2), we also get
\begin{equation}\label{upper bound}
\begin{split}
\int_{-\infty}^\infty \widehat{g}(\alpha) \, \FF_\Gamma(\alpha, T) \, \d\alpha  &\leq  \int_{-(1-\delta)\Delta}^{(1-\delta)\Delta} \widehat{g}(\alpha) \, \FF_\Gamma(\alpha,T) \, \d\alpha \\
&=  \int_{-(1-\delta)\Delta}^{(1-\delta)\Delta} \widehat{g}(\alpha) \,\d\nu(\alpha) +o(1)\leq \Phi_\nu(h)+ \frac{\varepsilon}{3} + o(1),
\end{split}
\end{equation}
as $T \to \infty$. So that, by Fourier inversion,
\begin{align}\label{eq:2}
   \nonumber \int_{-\infty}^\infty &\mathcal{G}_L \bigg(\alpha -b-\frac{l}{2}\bigg) \, \FF_\Gamma(\alpha, T)\, \d \alpha   \\
    \nonumber &=\bigg(\frac{\lambda T}{2\pi}\log\,{T}\bigg)^{-1} \operatorname{Re} \sum_{0<\gamma, \gamma'\leq T}T^{i\lambda(b+\frac{l}{2})(\gamma-\gamma')}(\mathscr{F}^{-1}[\chi_{[-L, L]}]\cdot g)\bigg( (\gamma\!-\!\gamma')\frac{\lambda\log\,{T}}{2\pi}\bigg) w(\gamma\!-\!\gamma')\\
    &\leq 2L \bigg(\frac{\lambda T}{2\pi}\log\,{T}\bigg)^{-1} \sum_{0<\gamma, \gamma'\leq T} g\bigg( (\gamma\!-\!\gamma')\frac{\lambda\log\,{T}}{2\pi}\bigg)w(\gamma\!-\!\gamma')\\
    \nonumber &= 2L \int_{-\infty}^\infty   \widehat{g}(\alpha) \, \FF_\Gamma(\alpha, T) \, \d\alpha \\
    \nonumber &\leq 2L\left\{\Phi_\nu(h)+ \frac{\varepsilon}{3} +o(1)\right\},
\end{align}
which is almost enough to establish the upper bound. It remains to show that the second integral in the right-hand side of \eqref{upper bound estimate} is bounded. But this follows since $|I_L|= 4M$, which is a finite length, and hence, $\int_{b + \frac{\ell}{2} + I_L}\FF_\Gamma(\alpha, T)\, \d \alpha \ll 4M$, by the argument of Goldston \cite[eq: (2.6)]{G}. Thus, for $\ell \geq \ell_0(\varepsilon)$, we get
\begin{align}\label{inequality 3}
 \frac{1}{\ell}\int_{b}^{b+\ell} \FF_\Gamma(\alpha, T) \, \d \alpha \leq  \frac{2 L}{\ell} \left\{ \Phi_\nu(h) + \frac{\varepsilon}{3} + o(1) \right\} \leq \mathbf{C}_\nu + \varepsilon + o(1).
 \end{align}

For the lower bound set $L=M+\beta$ and observe that by appropriately shifting in \eqref{inequality 1}
\begin{align*}
      \chi_{[-\beta, \beta]}(x) +  \|\widehat{h}\|_1\cdot \chi_{I_L}(x) \geq \mathcal{G}_L\left(x\right), 
\end{align*}
whence we get the inequality
\begin{align*}
    \int_{-\beta}^{\beta}\FF_\Gamma(\alpha, T) \d\alpha +  \|\widehat{h}\|_1\int_{I_L}\FF_\Gamma(\alpha, T) \,\d \alpha \geq \int_{\R} \mathcal{G}_L(\alpha)\,\FF_\Gamma(\alpha, T) \,\d\alpha. 
\end{align*}
Now, letting $m_\gamma$ be the multiplicity of the element $\gamma$  of the sequence (that is, $m_\gamma := \#\{n : \gamma_n = \gamma \}$), we deduce that
\begin{align*}
\int_{-\infty}^\infty \mathcal{G}_L(\alpha) & \FF_\Gamma(\alpha, T) \d\alpha  = \bigg(\frac{\lambda T}{2\pi}\log\,{T}\bigg)^{-1} \sum_{0<\gamma, \gamma'\leq T} g\bigg( (\gamma\!-\!\gamma')\frac{\lambda \log\,{T}}{2\pi}\bigg) \bigg( \frac{2L \sin{(L(\gamma\!-\!\gamma')\lambda\log\,{T})}}{L(\gamma\!-\!\gamma')\lambda\log\,{T}} \bigg) w(\gamma\!-\!\gamma')\\
&\geq 2L\bigg(\frac{\lambda T}{2\pi}\log\,{T}\bigg)^{-1} \bigg\{ \sum_{0<\gamma\leq T}m_\gamma +\min_{x\in \R}\frac{\sin{x}}{x} \sum_{\substack{0<\gamma, \gamma'\leq T\\ \gamma\neq \gamma'}} g\bigg( (\gamma\!-\!\gamma')\frac{\lambda \log\,{T}}{2\pi}\bigg) w(\gamma\!-\!\gamma')\bigg\} \\
& =2 L\bigg(\frac{\lambda T}{2\pi}\log\,{T}\bigg)^{-1}\bigg\{ (1-s_0) \sum_{0<\gamma\leq T}m_\gamma + s_0 \sum_{\substack{0<\gamma, \gamma'\leq T}}g\bigg( (\gamma\!-\!\gamma')\frac{\lambda\log\,{T}}{2\pi}\bigg)  w(\gamma\!-\!\gamma') \bigg\} \\
&\geq 2 L \bigg\{ \bigg(\frac{\lambda T}{2\pi}\log\,{T}\bigg)^{-1} (1-s_0) \sum_{0<\gamma\leq T} 1 + s_0 \int_{-\infty}^\infty \widehat{g}(\alpha) \, \FF_\Gamma(\alpha, T) \, \d\alpha   \bigg\}.
\end{align*}
Thus, recalling $s_0$ is negative, we apply \eqref{upper bound} and assumption (A1) to get 
\begin{align*}
    \int_{-\infty}^\infty \mathcal{G}_L(\alpha) \, \FF_\Gamma(\alpha, T)\, \d\alpha \geq 2L\left[1-s_0+s_0\Phi_\nu(h) + s_0\frac{\varepsilon}{3} +o(1)\right]  \geq 2\beta\left[1-s_0+s_0\Phi_\nu(h) + s_0\frac{\varepsilon}{3} + o(1)\right], 
\end{align*}
as $T\rightarrow \infty$. So that we have 
\begin{align}\label{eq:3}
\liminf_{T\rightarrow \infty}\frac{1}{2\beta}\int_{-\beta}^{\beta}\FF_\Gamma(\alpha, T)  \,\d\alpha \geq \left[1-s_0+s_0\Phi_\nu(h) + s_0\frac{\varepsilon}{3} +o(1)\right] -\frac{1}{2\beta}\limsup_{T\rightarrow \infty} \|\widehat{h}\|_1\int_{I_L}\FF_\Gamma(\alpha, T) \d \alpha.
\end{align}
Note that $|I_L|=4M$, a finite length, and thus the integral on the right-hand side of the inequality above is bounded. Therefore, by \eqref{eq:3} and our choice of $h$, for $\beta\geq \beta_0(\varepsilon)$ we have 
\begin{align*}
    \liminf_{T\rightarrow \infty}\frac{1}{2\beta}\int_{-\beta}^{\beta}\FF_\Gamma(\alpha, T) \d\alpha \geq 1-s_0+s_0(\mathbf{C}_\nu +\varepsilon).
\end{align*}

Also, when $\beta=b+\ell$, the evenness of $\FF_\Gamma(\alpha)$ yields 
\begin{equation}\label{eq:4}
\begin{split}
     \frac{1}{\ell}\int_{b}^{b+\ell}\FF_\Gamma(\alpha, T) \, \d\alpha & = \frac{2\beta}{2\ell}\frac{1}{2\beta}\int_{-\beta}^{\beta}\FF_\Gamma(\alpha, T) \, \d\alpha - \frac{1}{2\ell}\int_{-b}^{b}\FF_\Gamma(\alpha, T) \, \d\alpha\\
    & = \frac{1}{2\beta}\int_{-\beta}^{\beta}\FF_\Gamma(\alpha, T) \,\d\alpha + \frac{b}{\ell}\bigg(\frac{1}{2  \beta}\int_{-\beta}^{\beta}\FF_\Gamma(\alpha, T) \,\d\alpha  - \frac{1}{2b}\int_{-b}^{b}\FF_\Gamma(\alpha, T)\, \d\alpha\bigg).
\end{split}
\end{equation}
Note that $\int_{-\beta}^{\beta}\FF_\Gamma(\alpha, T) \,\d\alpha \ll \beta$, which makes the averages appearing in the equation above uniformly bounded. Hence, by using \eqref{eq:3} and \eqref{eq:4}, we have that for any $\varepsilon>0$ and given $b\geq 1$ there exists $\ell_0(b, \varepsilon)$ such that for all $\ell \geq \ell_0(b, \varepsilon)$
\begin{align*}
    \frac{1}{\ell}\int_{b}^{b+\ell}\FF_\Gamma(\alpha, T)\, \d\alpha \geq 1-s_0+s_0\mathbf{C}_\nu -\varepsilon + o(1),
\end{align*}
as $T \to \infty$, finishing the proof.
\subsection{Proof of Theorem \ref{thm:indptlowerbound}}\label{sec:univlowerbound} For any non-negative function $g \in L^1(\R)$ such that $\widehat{g} \in L^1(\R)$, we can write by the Fourier inversion formula 
\begin{align*}
    \int_{-\infty}^{\infty}\widehat{g}(\alpha)\, &\FF_\Gamma(\alpha, T)\, \d\alpha  = \bigg(\frac{\lambda T}{2\pi}\log\,{T}\bigg)^{-1} \sum_{0<\gamma, \gamma'\leq T} g\bigg( (\gamma\!-\!\gamma')\frac{\lambda \log\,{T}}{2\pi}\bigg) w(\gamma\!-\!\gamma')\\
    &= \bigg(\frac{\lambda T}{2\pi}\log\,{T}\bigg)^{-1} g(0)\sum_{0<\gamma\leq T} m_\gamma +\bigg(\frac{\lambda T}{2\pi}\log\,{T}\bigg)^{-1} \sum_{\substack{0<\gamma, \gamma'\leq T\\ \gamma\neq \gamma'}} g\bigg( (\gamma\!-\!\gamma')\frac{\lambda \log\,{T}}{2\pi}\bigg)w(\gamma\!-\!\gamma').
\end{align*}
Since the sum of off-diagonal terms is non-negative, we have by (A1) that, as $T \to \infty$,
\begin{align*}
  \int_{-\infty}^{\infty}\widehat{g}(\alpha)\,\FF_\Gamma(\alpha, T)\, \d\alpha \geq g(0) + o(1).
\end{align*}
Now if $g$ is such that $\widehat{g}(\alpha) \leq \chi_{[-\beta, \beta]} (\alpha)$, we have
\begin{align*}
    \int_{-\beta}^{\beta}\FF_\Gamma (\alpha, T)\,\d\alpha \geq \int_{-\infty}^{\infty}\widehat{g}(\alpha) \, \FF_\Gamma(\alpha, T)\,\d\alpha \geq g(0) + o(1),
\end{align*}
as $T\to\infty$, leading naturally to the extremal problem (EP2) defined in subsection \ref{EP}. 
We have thus just proven a lower bound depending on the Fourier optimization constant $\mathbf{D}_\beta$.
\begin{lemma}\label{lem:lowerboundep}
    Let $\Gamma$ be a sequence of real numbers satisfying (A1). The following bound holds as $T \to \infty$:
    \begin{align*}
        \int_{-\beta}^{\beta}\FF_\Gamma(\alpha, T) \, \d\alpha \geq \mathbf{D}_\beta + o(1).
    \end{align*}
\end{lemma}
We know how to solve problem (EP2) exactly. This is the next step towards the lower bound independent of $\Delta$ in Theorem \ref{thm:indptlowerbound}.

\begin{proposition}\label{prop:ep2solution}
    For all $\beta > 0$, we have the equality $\mathbf{D}_\beta = \beta$.
\end{proposition}

\begin{proof}
    With no loss of generality, we can assume $\beta = 1$. This is because for any $g \in \mathcal R_\beta$, one can rescale $g$ to get an $h \in \mathcal R_1$ by setting $h(x)= \beta\,  g(\beta x)$. This goes in both directions, plainly yielding $\mathbf{D}_\beta  = \beta \, \mathbf{D}_1$.
    
    If we pick $g_0(x) = \left(\frac{\sin \pi x}{\pi x} \right)^2$, we have $\widehat {g_0}(\alpha) = \left(1-|\alpha| \right)_+$, it is clear that $g_0 \in \mathcal R_1$, so that $\mathbf{D}_1 \geq 1$. Now we prove that we can do no better than that. Let $g \in \mathcal{R}_1$ and $\phi$ be a non-negative smooth function of compact support such that $\widehat{\phi} \geq 0$ and $\int \phi = 1$. If $\phi^\eps(x):= \frac{1}{\eps} \phi(\frac{x}{\eps}) $, the function $g_\eps = g * \phi^\eps$ is integrable by Young's convolution inequality, it is non-negative and, since $\widehat{g_\eps}(\alpha) =  \widehat g(\alpha) \, \widehat{\phi}(\eps \alpha)$, it also satisfies $\widehat{g_\eps}(\alpha) \leq \chi_{[-1, 1]} (\alpha)$.
    Note that $g_\eps$ has bounded variation, which can be seen through another application of Young's convolution inequality. We can therefore apply the Poisson summation formula to obtain
    \begin{align*}
        \sum_{n \in \Z } g_\eps(n) = \sum_{k \in \Z }\widehat{g_\eps}(k), 
    \end{align*}
    but since $g_\eps \geq 0$ and $\widehat{g_\eps} \leq 0$ outside of $[-1, 1]$, we get
     \begin{align*}
        g_\eps(0) \leq \sum_{n \in \Z } g_\eps(n) = \sum_{k \in \Z } \widehat{g_\eps} (k) \leq \widehat{ g_\eps} (0) = \widehat g (0) \, \widehat{\phi}(0) = \widehat g(0). 
    \end{align*}
Since $g$ is a continuous function, we have that $g_\eps(x) \to g(x)$ converges uniformly in compact sets as $\eps \to 0$, so that $g(0) \leq \widehat g (0) \leq 1 $, implying $\mathbf{D}_1 \leq 1$ because $g$ was arbitrary. We thus have $\mathbf{D}_1 = 1$, concluding the proof.
\end{proof}
Putting together Lemma \ref{lem:lowerboundep}, Proposition \ref{prop:ep2solution} and the identity \eqref{eq:4}, we get Theorem \ref{thm:indptlowerbound}.
\section{The Reproducing Kernel Hilbert Space \texorpdfstring{${\mathcal H}_{\nu}$}{Hnu}}\label{sec:rkhs}
{
\subsection{Proof of Theorem \ref{thm:rkhs}}
Taking the Fourier transform of the measure $\nu$ that is compactly supported in $[-\Delta, \Delta]$, we obtain
\begin{equation}\label{eq:nuhat}
\begin{split}
    \widehat{\nu}(x)= \, &c_1 + \frac{c_2 \, e^{-c_3 \Delta}}{(c_3^2 + 4 \pi^2 x^2)^2}\bigg\{2 e^{c_3\Delta}(c_3^2 -4 \pi^2 x^2) \\ &- 2( c_3^2 -4\pi^2x^2+ c_3^3 \Delta + 4 c_3 \pi^2 x^2 \Delta)\cos(2\pi\Delta x)
       +   4 \pi x (2c_3 + c_3^2 \Delta + 4 \pi^2 x^2\Delta)\sin(2\pi  \Delta x)\bigg\}. 
\end{split}
\end{equation}

If we prove that this function is bounded above and below away from zero, i.e., that there exists $a,b > 0$ such that $a^2 \leq \widehat{\nu}(x) \leq b^2$, we are done, because then $\widehat{\nu}(x) \, \d x$ is a positive measure equivalent to Lebesgue, and we have the inequalities  
\begin{equation*}
    a^2 \int_\R g(x) \, \d x \leq \int_\R g(x)\, \widehat{\nu}(x) \, \d x \leq  b^2 \int_\R g(x) \, \d x
\end{equation*}
for all Lebesgue measurable $g \geq 0$, whence we also have \eqref{eq:equivalentnorm} by taking $g = |f|^2$ for $f \in L^2(\R)$. In particular, this means $\mathcal H_{\nu}$ is a normed linear space, and its completeness follows from the completeness of $PW(\pi \Delta)$. This likewise applies to the reproducing kernel property.

Now, we can trivially estimate it from above
\begin{equation*}
    | \widehat{\nu}(x)| \leq c_1 +  c_2 \int_{-\Delta}^{\Delta} \lvert\alpha\rvert\ e^{-c_3 \lvert\alpha\rvert } d\alpha \leq c_1 + c_2 \Delta^2,     
\end{equation*}
and to find a lower bound, set $t := 2 \pi \Delta x$ and $\sigma := c_3 \Delta$ and write
\begin{equation*}
    \widehat{\nu}(x) = c_2 \Delta^2 \left\{ \frac{c_1}{\Delta^2 c_2} - G(\sigma,t) \right\},
\end{equation*}
where
\begin{align*}
    G(\sigma,t)= - \frac{e^{-\sigma}}{(\sigma^2 +t^2)^2}\bigg\{2 e^{\sigma}(\sigma^2 - t^2) -  2( \sigma^2 - t^2 + \sigma^3 + t^2 \sigma)\cos(t)
       +  2 t (2 \sigma  + \sigma^2 + t^2 )\sin(t) \bigg\}.
\end{align*}
We will now show that $G(\sigma, t)$ cannot be too large. 

Observe first that this is a harmonic function. In fact, it is the real part of a holomorphic function of the variable $z = -\sigma + it$. Explicitly,
\begin{equation*}
    G(\sigma, t) = \operatorname{Re} \left\{ \frac{e^z (1-z)}{z^2} - \frac{1}{z^2}\right\}.
\end{equation*}
Hence, if we let $\Omega_R$ be the region $\{-\sigma + i t: \sigma > 0 , t > 0, |z| < R \}$ , we have by the maximum principle that
\begin{equation*}
   \max_{z \in \overline{\Omega}_R} G(\sigma, t) \leq \max_{z \in \partial\Omega_R}  \operatorname{Re} \left\{ \frac{e^z (1-z)}{z^2} - \frac{1}{z^2}\right\},
\end{equation*}
but, when $|z| = R$, we have that
\begin{equation*}
    \left| \operatorname{Re} \left\{ \frac{e^z (1-z)}{z^2} - \frac{1}{z^2}\right\} \right| \leq \frac{e^{-\sigma}(1+R)+1}{R^2},
\end{equation*}
which becomes very small when $R \to \infty$. Moreover, when $t = 0$, we have
\begin{equation*}
    G(\sigma, 0) = - \frac{e^{-\sigma}}{\sigma^4}\bigg\{2 e^{\sigma}(\sigma^2) -  2( \sigma^2 + \sigma^3)\bigg\} = \frac{2}{\sigma^2}\left\{ - 1  + e^{-\sigma} \left(1 + \sigma \right)\right\} \leq 0
\end{equation*}
by the elementary inequality $1 + \sigma \leq e^\sigma$ for $\sigma \geq 0$. So that to bound $G$ in the first quadrant, it is enough to consider it on the half-line $\sigma=0$, $t \geq 0$. That is,
\begin{equation*}
    \sup_{\sigma \geq 0, t \geq 0} G(\sigma, t) = \sup_{t > 0}  \frac{  2 -  2\cos(t)
       -  2 t\sin(t)}{t^2} = 0.586\ldots.
\end{equation*}
Thus, choosing $\frac{c_1}{\Delta^2 c_2} \geq \frac{3}{5} > 0.586\ldots$, we obtain that $\widehat{\nu}$ is bounded below uniformly away from zero, concluding the proof.
\begin{remark}
By a standard procedure involving a Fourier uncertainty principle (as used in \cite[proof of Lemma 12]{CCLM}), we can extend the range of admissible values of $c_1, c_2,$ and $\Delta$ all the way up to $\frac{c_2}{c_1} \Delta^2 \leq \frac{1}{ 0.586\ldots}$ to guarantee that $(\mathcal H_{ \nu}, \| \cdot \|_{\nu})$ is a reproducing kernel Hilbert space. 
\end{remark}

\subsection{Proof of Lemma \ref{lem:rkbound}}
It is immediate from the inequalities in \eqref{eq:equivalentnorm} that $\mathcal{H}_\nu$ is a reproducing kernel Hilbert space which is equal to $PW(\pi \Delta)$ as a set.

Returning to (EP1), we recall that, to find $\mathbf{C}_\nu$, we want to minimize
\begin{align*}
 \frac{\Phi_\nu(g)}{g(0)} = \frac{1}{g(0)} \int_{-\Delta}^{\Delta}\widehat{g}(\alpha)\,\d\nu(\alpha),
\end{align*}
over $g$ in the class $\mathcal A_\Delta$. We restrict ourselves to a subclass of bandlimited functions to obtain an upper bound for $\mathbf{C}_\nu$. This is an adaptation of the proof of \cite[Corollary 14]{CCLM} to our setting.

Denote by $\B_\Delta:= \{g \in \A_\Delta: \supp{(\widehat{g})}\subseteq [-\Delta, \Delta] \}$ and let $g\in \B_\Delta\subset \A_\Delta$. Since $\supp{(\widehat{g})}\subseteq [-\Delta, \Delta]$, we know that $g$ has a unique extension to an entire function of exponential type at most $2\pi\Delta$, which we also denote by $g$. Since $g(x)\geq 0$ and $g\in L^1(\mathbb{R})$ by Krein's decomposition theorem \cite[p. 154]{A} there exists an entire function $f\in PW(\pi \Delta)$ such that 
    \begin{equation*}
        g(z)=f(z)\overline{f(\bar z)}
    \end{equation*}
    for all $z \in \C$. By \eqref{eq:equivalentnorm}, $f$ is also a member of $\mathcal H_{\nu}$. Applying Plancherel's formula,
    \begin{equation*}
         \int_{-\Delta}^{\Delta}\widehat{g}(\alpha)\,\d\nu(\alpha) =  \int_\R g(x) \,  \widehat{\nu}(x) \, \d x = \int_\R |f(x)|^2 \, \widehat{\nu}(x) \, \d x. 
    \end{equation*}
    On the other hand, by the reproducing kernel property and the Cauchy--Schwarz inequality,
    \begin{equation*}
        g(0) = |f(0)|^2 = |\langle f, K_\nu (0, \cdot) \rangle|^2 \leq K_\nu(0,0)  \int_\R |f(x)|^2 \, \widehat{\nu}(x) \, \d x,
    \end{equation*}
    so that 
    \begin{equation*}
         \frac{1}{g(0)} \int_{-\Delta}^{\Delta}\widehat{g}(\alpha)\,\d\nu(\alpha) \geq \frac{1}{K_\nu(0,0)}.
    \end{equation*}
    Now, equality can be attained in the above inequality if and only if $g$ is a multiple of $K_\nu(0,z)\overline{K_\nu(0,\bar z)}$, which occurs in $\mathcal{B}_\Delta$ since $K_\nu(0,\cdot) \in L^2(\R)$ by \eqref{eq:equivalentnorm}. This is the best bound we can obtain for $\mathbf{C}_\nu$ working within the class $\B_\Delta$, and it gives us what we had set out to prove.

}

\section{Computing Reproducing Kernels}\label{sec:computing-rks}
In this section, we will study the reproducing kernel $K_\nu:\C^2 \to \C$ of the space $\mathcal H_{ \nu}$,
providing proofs for Theorems \ref{thm:k00firstcase} and \ref{thm:k00secondcase}. The starting point for both proofs is the same. For all $f \in \mathcal H_{\nu}$ and $w \in \C$, we must have by the reproducing kernel property and Plancherel's formula that
\begin{equation*}
    f(w)=\int_{\mathbb{R}}f(x)\overline{K_\nu(w, x)} \, \widehat{\nu}(x)\,\d x=\int_{-\Delta/2}^{\Delta/2} \widehat{f}(\xi) \mathscr{F}^{-1}\left[\overline{K_\nu(w, \cdot)} \, \widehat{\nu}\,\right](\xi)\,\d\xi.
\end{equation*}
But by Fourier inversion
\begin{equation*}
    f(w) = \int_{-\Delta/2}^{\Delta/2} \widehat{f}(\xi) e^{2\pi i w \xi}\, \d\xi
\end{equation*}
as well. Since $\mathscr{F}(\mathcal H_{ \nu}) = \mathscr{F}(PW(\pi \Delta)) = L^2([-\frac{\Delta}{2},\frac{\Delta}{2}])$ by the Paley--Wiener theorem, we must therefore have
\begin{equation*}
    \int_{\mathbb{R}}\overline{K_\nu(w, x)} \,  \widehat{\nu}(x) e^{2\pi i \xi x}\, \d x=  \mathscr{F}^{-1}\left[\overline{K_\nu(w, \cdot)} \, \widehat{\nu}\,\right](\xi) = e^{2\pi i w \xi}
\end{equation*}
for almost all $\xi \in [-\frac{\Delta}{2},\frac{\Delta}{2}]$. Denoting  by $u_w=\mathscr{F}\left[\overline{K_\nu(w, \cdot)}\right]$, we will thus seek to solve for $u_w$ in the $L^2([-\frac{\Delta}{2},\frac{\Delta}{2}])$ equation given by
\begin{equation}\label{eq:functionaleq}
    (u_w * \nu)(\xi) = \int_\R u_w(\xi -\alpha) \, \d \nu(\alpha) =  e^{-2\pi i w \xi}
\end{equation}
for almost all $\xi \in [-\frac{\Delta}{2},\frac{\Delta}{2}]$.

\subsection{Proof of Theorem \ref{thm:k00firstcase}} We obtain the expression for the reproducing kernel $K_\nu(w,z)$ of the space $\mathcal H_{\nu}$ with weight given by \eqref{eq:nuhat}, which in this case simplifies to
\begin{align*}
    \widehat{\nu}(x)= c_1 + c_2 \Delta \frac{\sin{2\pi \Delta x}}{\pi x} - c_2\left(\frac{\sin{\pi \Delta x}}{\pi x}\right)^2.
\end{align*}
To state our result, we define the functions
\begin{align*}
    a(w) &:= \frac{-2 c_2\left(\cos{\left(\pi \Delta w\right)}+\pi \Delta w \sin{\left(\pi \Delta w\right)}\right)}{c_1(2c_1\pi^2w^2-c_2)\left(2\cos{\left(\sqrt{\frac{c_2}{2c_1}}\Delta\right)}+\sqrt{\frac{2c_2}{c_1}} \Delta\sin{\left(\sqrt{\frac{c_2}{2c_1}}\Delta\right)}\right)} ;\\
    b(w) &:= \sqrt{\frac{2c_2}{c_1 }}\cdot\frac{i \pi w \cos{\left(\pi \Delta w\right)} }{\left(2c_1\pi^2w^2-c_2\right)\cos{\left(\sqrt{\frac{c_2}{2 c_1}}\Delta\right)}};\\
    c(w) &:=  \frac{2\pi^2w^2}{2c_1\pi^2w^2-c_2},
\end{align*}
and
\begin{align*}
    q(z) &:=  \frac{\sqrt{2 c_1 c_2}}{c_2 - 2 c_1 \pi^2 z^2}\sin\left(\sqrt{\frac{c_2}{2c_1}} \Delta\right) \cos(\pi \Delta z) - \frac{ 2 \pi c_1 z}{c_2 - 2 c_1\pi^2 z^2}\cos\left(\sqrt{\frac{c_2}{2c_1}} \Delta\right) \sin(\pi \Delta z);\\
    r(z) &:=  i \frac{\sqrt{2 c_1 c_2}}{c_2 - 2 c_1 \pi^2 z^2}\cos\left(\sqrt{\frac{c_2}{2c_1}} \Delta\right) \sin(\pi \Delta z) + i \frac{ 2 \pi c_1 z}{c_2 - 2 c_1\pi^2 z^2}\sin\left(\sqrt{\frac{c_2}{2c_1}} \Delta\right) \cos(\pi \Delta z);.
\end{align*}
We will show a stronger result than Theorem \ref{thm:k00firstcase} in the following theorem.

\begin{theorem}\label{thm:fullexpressionkwz}
Under the hypotheses of the Theorem \ref{thm:k00firstcase}  we have
\begin{equation*}
    K_\nu(w,z) = a(\bar w)  q(z) + b(\bar w) r(z) + c(\bar w) \frac{\sin\pi \Delta(z - \bar w)}{\pi(z - \bar w)},
\end{equation*}
where the above expression is understood in the limit sense when either $z$ or $w$ equal $\pm \frac{1}{\pi} \sqrt{\frac{c_2}{2 c_1}}$.
\end{theorem}
In particular,
\begin{align*}
    K_\nu(0,0) = a(0) \, q(0) =  \sqrt{ \frac{2 }{c_1 c_2}} \frac{\sin{\left(\sqrt{\frac{c_2}{2c_1}} \Delta \right)}}{\cos{\left( \sqrt{\frac{c_2}{2c_1}} \Delta \right)}+\sqrt{\frac{c_2}{2c_1}} \Delta \sin{\left(\sqrt{\frac{c_2}{2c_1}} \Delta \right)}}.
\end{align*}
Thus, Theorem \ref{thm:k00firstcase} follows as a consequence of this result.

The proof of Theorem \ref{thm:fullexpressionkwz} is obtained by solving the functional equation \eqref{eq:functionaleq}. We collect the results about this equation in the following lemma.

\begin{lemma}\label{lem:functionalodesimple}
    Let $c_1,c_2 > 0$, $w \in \C$,  $\Delta > 0$. When $\frac{c_2}{c_1}\Delta^2 < 2 $, the following statements hold.
    \begin{enumerate}[(i)]
        \item There exists a unique function $u_w \in L^2(\R)$ with $\mathrm{supp} \, u_w \subseteq [-\frac{\Delta}{2},\frac{\Delta}{2}]$ such that
        \begin{equation*}
            c_1 u_w (\xi) + c_2 \int_{-\Delta}^\Delta u_w(\xi - \alpha) |\alpha| \, d \alpha = e^{- 2 \pi i \xi w} 
        \end{equation*}
        for almost all  $\xi \in [-\frac{\Delta}{2},\frac{\Delta}{2}]$. 
        \item When $w \neq \pm \frac{1}{\pi} \sqrt{\frac{c_2}{2 c_1}}$, the function $u_w$ is given by the expression
        \begin{equation}\label{eq:ftrk}
            u_w(\xi)=a(w) \cos \left(\sqrt{\frac{2c_2}{c_1}}\xi\right)+ b(w) \sin \left(\sqrt{\frac{2c_2}{c_1}}\xi\right)+ c(w)e^{-2\pi i w \xi } 
        \end{equation}
        in the interval $[-\frac{\Delta}{2}, \frac{\Delta}{2}]$ and is identically zero outside it. 

        \item If we define the entire function
        \begin{equation*}
            k_w (z) := \int_{-\Delta/2}^{\Delta/2} u_w(\alpha) e^{2 \pi i \alpha z} \, \d\alpha,
        \end{equation*}
        then $k_w \in PW(\pi \Delta)$ and
        \begin{equation*}
            f(w) = \int_\R f(x) \, k_w(x) \, \widehat{\nu}(x) \, \d x
        \end{equation*}
        for all $f \in PW(\pi \Delta)$.
    \end{enumerate}
    
\end{lemma}

\begin{proof}
    Henceforth, we will use the abbreviation $I = [-\frac{\Delta}{2},\frac{\Delta}{2}]$. In the first part, we are asserting the existence and uniqueness  of solutions to the functional equation
    \begin{equation}\label{eq:operatorfunctionaleq}
        u + T u = \frac{1}{c_1} e^{-2 \pi i w \xi} 
    \end{equation}
    in $L^2(I)$,  where $T: L^2(I) \to L^2(I)$ is the operator given by
    \begin{equation*}
        (T u) (\xi) := \frac{c_2}{c_1} \int_{-\Delta/2}^{\Delta/2} u(\alpha) |\xi - \alpha| \, \d \alpha.
    \end{equation*}
    Now, $T$ is a Hilbert--Schmidt operator, since the kernel $S(x,y) = |x-y|$ is in $L^2(I\times I)$. The operator $T$ is therefore compact so, by Fredholm's alternative, there exists a unique solution to (\ref{eq:operatorfunctionaleq}) if and only if the equation
    \begin{equation*}
         u + T u = \mathbf{0}
    \end{equation*}
    only admits the trivial solution. Suppose $u \in L^2(I)$ is a solution to the above equation, i.e.,
    \begin{equation*}
        u(\xi) = - \frac{c_2}{c_1} \int_{-\Delta/2}^{\Delta/2} u(\alpha) |\xi - \alpha| \, \d \alpha
    \end{equation*}
    for a.e. $\xi \in I$. It follows that $u \in L^\infty(I)$ by, say, the Cauchy--Schwarz inequality, and that
    \begin{equation*}
        |u(\xi)| \leq \| u\|_{L^{\infty}(I)} \frac{c_2}{c_1} \int_{-\Delta/2}^{\Delta/2} |\xi - \alpha| \, \d \alpha
    \end{equation*}
    for a.e. $\xi$, yielding
    \begin{equation*}
        \| u\|_{L^{\infty}(I)} \leq  \| u\|_{L^{\infty}(I)} \, \frac{c_2}{c_1} \, \sup_{\xi \in I} \int_{-\Delta/2}^{\Delta/2} |\xi - \alpha| \, \d \alpha = \frac{c_2}{c_1} \frac{\Delta^2}{2}\| u\|_{L^{\infty}(I)},
    \end{equation*}
    which, for the admissible values of $c_1, c_2,\text{ and } \Delta$, is only possible when $u$ is trivial, proving (i).

    Now, formally differentiating equation \eqref{eq:functionaleq}, we get
    \begin{equation}\label{eq:derivfunctionaleq}
        c_1 u'(\xi) + c_2 \int_{-\Delta}^{\Delta} u'(\xi-\alpha) \lvert \alpha\rvert\, \d\alpha = -2\pi i w e^{-2\pi i w \xi}.
    \end{equation}
    Doing it once more yields the ordinary differential equation
    \begin{equation}\label{eq:1stode}
         c_1 u''(\xi) + 2 c_2 u(\xi) = -4\pi^2w^2 e^{-2\pi i w \xi}.
    \end{equation}
    In turn, a solution to the above ODE for $\xi \in I$, subject to the boundary conditions (obtained by plugging in $\xi = 0$ in the expressions \eqref{eq:functionaleq} and \eqref{eq:derivfunctionaleq})
    \begin{align}
        \begin{cases}\label{eq:1stodeBC}
            c_1 u(0) + c_2 \int_{-\Delta/2}^{\Delta/2} u(\alpha) |\alpha| \, \d \alpha = 1; \\
            c_1 u'(0) - c_2 \int_{-\Delta/2}^{\Delta/2} u(\alpha) \operatorname{sgn}{(\alpha)}\, \d \alpha = -2 \pi i w,
        \end{cases}
    \end{align}
    is a solution to the functional equation and, therefore, its only solution. Indeed, integrating the ODE \eqref{eq:1stode} from $0$ to a $\xi \in (0,\frac{\Delta}{2})$,
    \begin{align*}
        c_1 u'(\xi) + 2 c_2 \int_0^\xi u(\alpha) \, \d\alpha = -2 \pi i w e^{-2 \pi i w \xi} + 2 \pi i w  + c_1 u'(0)
    \end{align*}
    and, integrating again,
    \begin{align*}
        c_1 u(\xi) + 2 c_2 \int_0^\xi \int_0^\alpha u(\beta) \, \d\beta \d\alpha = e^{-2 \pi i w \xi} - 1 + 2 \pi i w \xi + c_1 u(0) + c_1 u'(0) \xi 
    \end{align*}
    which, by Fubini's theorem, is
    \begin{equation*}
        c_1 u(\xi) + 2 c_2 \int_0^\xi (\xi - \beta) u(\beta) \, \d\beta = e^{-2 \pi i w \xi} - 1 + 2 \pi i w \xi  + c_1 u(0) + c_1 u'(0) \xi.
    \end{equation*}
    Applying the boundary conditions, we get
    \begin{align*}
        c_1 u(\xi) + 2 c_2 \int_0^\xi (\xi - \beta) u(\beta) \, \d\beta = e^{-2 \pi i w \xi}  -  c_2 \int_{-\Delta/2}^{\Delta/2} u(\alpha) |\alpha| \, \d \alpha + c_2 \int_{-\Delta/2}^{\Delta/2} \xi u(\alpha) \operatorname{sgn}{(\alpha)} \d \alpha,
    \end{align*}
    which, after rearranging, is exactly
    \begin{equation*}
        c_1 u(\xi) + c_2 \int_{-\Delta/2}^{\Delta/2} u(\alpha) |\xi - \alpha| \, \d \alpha =  e^{-2\pi i w \xi},
    \end{equation*}
    and the same can be done for negative $\xi$, proving the assertion that the solution of the ODE is a solution to the functional equation. 
    
    Now it remains to be shown that the expression in (\ref{eq:ftrk}) solves the ODE with the boundary conditions \eqref{eq:1stodeBC} to conclude the proof of (ii). But a solution to \eqref{eq:1stode} must be of the form
    \begin{equation*}
        u(\xi)=A\cos{\left(\sqrt{\frac{2c_2}{c_1}}\xi\right)} +B\sin{\left(\sqrt{\frac{2c_2}{c_1}}\xi\right)} + \frac{2\pi^2w^2}{2c_1\pi^2w^2-c_2}e^{-2\pi i w \xi}
    \end{equation*}
    for all $\xi \in I$ for some coefficients $A$ and $B$. We use \eqref{eq:1stodeBC} to identify the coefficients, whence it follows $A = a(w)$ and $B = b(w)$ exactly.

     Now suppose $f \in PW(\pi \Delta)$ and let $k_w$ be defined as before. Since $u_w$ is in $L^2(I)$, it is clear that $k_w \in PW(\pi \Delta)$. By applying Plancherel's formula and Fourier inversion,
    \begin{align*}
        \int_{\mathbb{R}}f(x) &k_w(x) \, \widehat{\nu}(x) \,\d x = \int_{-\Delta/2}^{\Delta/2} \widehat{f}(\xi) \, \mathscr{F}^{-1} [k_w \cdot \widehat{\nu} \,] (\xi) \, \d \xi \\
        &= \int_{-\Delta/2}^{\Delta/2} \widehat{f}(\xi) \left\{ c_1 u_w(-\xi) + c_2 \int_{-\Delta}^{\Delta} u_w(-\xi-\alpha) \lvert\alpha\rvert\, d\alpha \right\} \, \d \xi \\
        &= \int_{-\Delta/2}^{\Delta/2} \widehat{f}(\xi) \,  e^{2 \pi i \xi w}  \, \d \xi \\
        &= f(w),
    \end{align*}
    proving item (iii) of the lemma. 
\end{proof}
As per the lemma, Theorem \ref{thm:fullexpressionkwz} will now follow by taking the inverse Fourier transform of \eqref{eq:ftrk} to obtain
\begin{equation*}
   k_w(z) = a(w)  q(z) + b(w) r(z) + c(w) \frac{\sin\pi \Delta(z -  w)}{\pi(z - w)},
\end{equation*}
and taking $K_\nu(w,z) = \overline{ k_w(\bar z)}$, which gives the desired expression whenever $w \neq \pm \frac{1}{\pi} \sqrt{\frac{c_2}{2 c_1}}$. However, we know that $K_\nu(w,z)$ always exists when $\frac{c_2}{c_1}\Delta^2 \leq 5/3$ and that
\begin{equation*}
    K_\nu(w,z) = \langle K_\nu(w, \cdot), K_\nu(z, \cdot)  \rangle = \overline{ \langle K_\nu(z, \cdot),  K_\nu(w, \cdot)  \rangle} =  \overline{K_\nu(z, w)},
\end{equation*}
which means that, since $K_\nu$ is holomorphic in the first variable, it must be anti-holomorphic in the second variable. In particular, the limit as $w \to \pm \frac{1}{\pi} \sqrt{\frac{c_2}{2 c_1}}$ of
\begin{equation*}
    a(\bar w)  q(z) + b(\bar w) r(z) + c(\bar w) \frac{\sin\pi \Delta(z - \bar w)}{\pi(z - \bar w)}
\end{equation*}
must always exist, which yields the full theorem.

\subsection{Proof of Theorem \ref{thm:k00secondcase}} It follows from the hypothesis of the Theorem \ref{thm:k00secondcase} that we can apply Theorem \ref{thm:rkhs}. Hence, $\mathcal{H}_\nu$ is a reproducing kernel Hilbert space. 

Let $\eta_1$ and $\eta_2$ be the solutions of the equation
\begin{equation}\label{eq:characteristic}
     \eta^4 +2\left(\frac{c_2}{c_1}- c_3^2\right)\eta^2+c_3^2\left(2\frac{c_2}{c_1} + c_3^2\right)=0,
\end{equation}
such that $\eta_1+\eta_2\neq 0$.  Without loss of generality we assume $\operatorname{Re}(\eta_1)\geq 0$ and $\operatorname{Re}(\eta_2)\geq 0$. Also, define 
\begin{equation*}
    \mu :=\frac{c_3^2}{2c_2+c_3^2 c_1}
\end{equation*}
and the following auxiliary function
\begin{equation*}
    C(\eta,z):=\frac{4\pi z \sin{(\pi \Delta z)}\cosh{\frac{\Delta \eta}{2}}+2\eta\cos{(\pi \Delta z)\sinh{\frac{\Delta \eta}{2}}}}{4\pi^2z^2+\eta^2}.
\end{equation*}
Note that, for any fixed $\eta$, the values $C(\eta, \pm i\frac{\eta}{2\pi})$ are understood as the limits $\lim\limits_{z\to \pm i\frac{\eta}{2\pi}}C(\eta, z)$.
\begin{theorem}\label{thm:reprokernelkzero}
   Under the hypotheses of Theorem \ref{thm:k00secondcase} we have the reproducing kernel
   \begin{equation}\label{eq:reprokernelkzero}
   \begin{split}
        K_\nu(0, z)=\left(\frac{1}{c_1}-A(0)\mu\right)&\frac{C(\eta_1, z)B(\eta_2)-C(\eta_2, z)B(\eta_1)}{A(\eta_1)B(\eta_2)-A(\eta_2)B(\eta_1)}\\&+\left(\frac{c_3^2}{c_1}+B(0)\mu\right)\frac{C(\eta_1, z)A(\eta_2)-C(\eta_2, z)A(\eta_1)}{A(\eta_1)B(\eta_2)-A(\eta_2)B(\eta_1)}
        +\mu\frac{\sin{(\pi \Delta z)}}{\pi z} 
   \end{split}
   \end{equation}
of the space $\mathcal{H}_\nu$, where the expression is interpreted as a limit whenever $\eta_1=\eta_2$.
\end{theorem}

In particular, when $z = 0$ we have \eqref{eq:k00exp}, concluding the proof of Theorem \ref{thm:k00secondcase}.
The proof of Theorem \ref{thm:reprokernelkzero} is analogous to the proof of Theorem \ref{thm:fullexpressionkwz} and it follows from the lemma below.
\begin{lemma}\label{lem:functionalodeexp}
    Let $c_1,c_2,c_3 > 0$, $w \in \C$, and  $\Delta > 0$. When $\frac{c_2}{c_1}\Delta^2 < 2 $ the following statements hold.
    \begin{enumerate}[(i)]
        \item There exists a unique function $u_w \in L^2(\R)$ with $\mathrm{supp} \, u_w \subseteq [-\frac{\Delta}{2},\frac{\Delta}{2}]$ such that
        \begin{equation}\label{eq:integralequationforexponential}
           c_1 u_w(\xi) + c_2 \int_{-\Delta}^{\Delta} u_w(\xi-\alpha) \lvert\alpha\rvert\ e^{-c_3 \lvert\alpha\rvert } \d\alpha = e^{- 2\pi i w \xi} 
        \end{equation}
        for almost all  $\xi \in [-\frac{\Delta}{2},\frac{\Delta}{2}]$. 
        \item The function $u_w$ is a solution of the ODE
        \begin{equation*}
                c_1 u^{(4)}(\xi) + 2(c_2 - c_1 c_3^2)u''(\xi) + (2 c_2 c_3^2 + c_1 c_3^4) u(\xi) =  (4 \pi^2 w^2 + c_3^2)^2 e^{-2 \pi i w \xi},
        \end{equation*} 
        subject to the conditions
        \begin{equation*}
            \begin{cases}
                c_1 u(0) + c_2 \int_{-\Delta/2}^{\Delta/2} u(\alpha)|\alpha|e^{-c_3 |\alpha|} \, \d\alpha = 1 ;\\
                c_1 u'(0) - c_2 \int_{-\Delta/2}^{\Delta/2} u(\alpha) \operatorname{sgn} \alpha \, e^{-c_3 |\alpha|}\, \d\alpha + c_2 c_3 \int_{-\Delta/2}^{\Delta/2} u(\alpha) \alpha \, e^{-c_3 |\alpha|} \, \d\alpha  = -2 \pi i w ;\\
                c_1 u''(0) + (2 c_2 - c_1 c_3^2 ) u(0) - 2 c_2 c_3 \int_{-\Delta/2}^{\Delta/2} u(\alpha) e^{-c_3 |\alpha|} \, \d\alpha = -4 \pi^2 w^2 - c_3^2;\\
                c_1 u'''(0) - (c_1 c_3^2 - 2 c_2) u'(0) -  2 c_2 c_3^2 \int_{-\Delta/2}^{\Delta/2} u(\alpha) \operatorname{sgn} \alpha \, e^{-c_3 |\alpha|}\, \d\alpha = (2 \pi i w)(4 \pi^2 w^2 + c_3^2),
            \end{cases}
        \end{equation*}
        in the interval $[-\frac{\Delta}{2},\frac{\Delta}{2}]$ and identically zero outside it.
        \item  If we define the entire function
            \begin{equation*}
                k_w(z) := \int_{-\Delta/2}^{\Delta/2} u_w(\alpha) e^{2 \pi i \alpha z} \, \d\alpha,
            \end{equation*}
            then $k_w \in PW(\pi \Delta)$ and
        \begin{equation*}
            f(w) = \int_\R f(x) \, k_w(x) \, \widehat{\nu}(x)\, \d x
        \end{equation*}
        for all $f \in PW(\pi \Delta)$.
    \end{enumerate}
\end{lemma} 
\begin{proof}
    Throughout we will use the abbreviation $I = [-\frac{\Delta}{2},\frac{\Delta}{2}]$. The proof of items (i) and (iii) are analogous to that of Lemma \ref{lem:functionalodesimple}. Now knowing that there is a solution $u_w$ to \eqref{eq:integralequationforexponential} in $L^2(I)$, we can quickly see $u_w \in C^\infty (I)$, since \eqref{eq:integralequationforexponential} is equivalent to
    \begin{align}
        \nonumber c_1 u_w(\xi) &= e^{- 2\pi i w \xi}  - c_2 \int_{-\Delta}^{\Delta} u_w(\xi-\alpha) \lvert\alpha\rvert\ e^{-c_3 \lvert\alpha\rvert } \d\alpha \\
        \nonumber&= e^{- 2\pi i w \xi}  - c_2 \int_{-\Delta/2}^{\Delta/2} u_w(\alpha) \lvert\xi-\alpha\rvert\ e^{-c_3 \lvert\xi-\alpha\rvert } \d\alpha \\
        \label{eq:diffexp}&= e^{- 2\pi i w \xi} - c_2 \int_{-\Delta/2}^{\xi} u_w(\alpha) (\xi-\alpha) e^{-c_3 (\xi-\alpha) } \d\alpha + c_2\int_{\xi}^{\Delta/2} u_w(\alpha) (\xi-\alpha) e^{c_3 (\xi-\alpha) } \d\alpha,
    \end{align}
    so that it is immediate that the integrability of $u_w$ implies that it is absolutely continuous in $I$. But then if $u_w$ is absolutely continuous, the same equation tells us it is $C^1$ with an absolutely continuous derivative by the fundamental theorem of calculus. Proceeding inductively, it follows that $u_w$ is smooth in the interval $I$.

    To obtain the ODE, it is enough to take derivatives of the expression \eqref{eq:diffexp}. By differentiating and collecting the appropriate terms, we get
    \begin{equation}\label{eq:intgeq2}
        c_1 u_w'(\xi) = -2 \pi i w e^{-2 \pi i w \xi} - c_2 \int_{-\Delta/2}^{\Delta/2} u_w(\alpha) \operatorname{sgn}(\xi-\alpha) e^{-c_3 \lvert\xi-\alpha\rvert } \d\alpha + c_2 c_3 \int_{-\Delta/2}^{\Delta/2} u_w(\alpha) (\xi-\alpha) e^{-c_3 \lvert\xi-\alpha\rvert } \d\alpha.
    \end{equation}
    Doing it again,
    \begin{equation}\label{eq:intgeq3}
        c_1 u_w''(\xi) = (-4 \pi^2 w^2  - c_3^2) e^{-2 \pi i w \xi} + (c_1 c_3^2 - 2 c_2)u_w(\xi) + 2 c_2 c_3\int_{-\Delta/2}^{\Delta/2} u_w(\alpha)  e^{-c_3 \lvert\xi-\alpha\rvert } \d\alpha,
    \end{equation}
    a third time,
    \begin{equation}\label{eq:intgeq4}
        c_1 u_w'''(\xi) = (2 \pi i w)(4 \pi^2  w^2 + c_3^2) e^{-2 \pi i w \xi} + (c_1 c_3^2 - 2 c_2)u_w'(\xi) - 2 c_2 c_3^2 \int_{-\Delta/2}^{\Delta/2} u_w(\alpha) \operatorname{sgn}(\xi - \alpha) \,   e^{-c_3 \lvert\xi-\alpha\rvert } \d\alpha,
    \end{equation}
    and a last time
    \begin{equation*}
        c_1 u_w^{(4)}(\xi) = (4 \pi^2  w^2 + c_3^2)^2 e^{-2 \pi i w \xi} + 2 (c_1 c_3^2 - c_2) u_w''(\xi) - (c_1 c_3^2 + 2 c_2)c_3^2 u_w(\xi), 
    \end{equation*}
    which is the desired equation. The boundary conditions are recovered by plugging in $\xi =0$ in the integral equations \eqref{eq:diffexp}, \eqref{eq:intgeq2}, \eqref{eq:intgeq3}, and \eqref{eq:intgeq4}, which $u_w$ and its derivatives must satisfy.
\end{proof}

\begin{proof}[Proof of Theorem \ref{thm:reprokernelkzero}]
Applying item (ii) of Lemma \ref{lem:functionalodeexp} to the case $w = 0$, we obtain that, in this case, the function $u = u_0$ is the solution to the ODE
\begin{equation}\label{eq:odeforthepositivec3}
                c_1 u^{(4)}(\xi) + 2(c_2 - c_1 c_3^2)u''(\xi) + (2 c_2 c_3^2 + c_1 c_3^4) u(\xi) =   c_3^4,
\end{equation} 
subject to the conditions
\begin{equation*}
    \begin{cases}
        c_1 u(0) + c_2 \int_{-\Delta/2}^{\Delta/2} u(\alpha)|\alpha|e^{-c_3 |\alpha|} \, \d\alpha = 1 ;\\
        c_1 u'(0) - c_2 \int_{-\Delta/2}^{\Delta/2} u(\alpha) \operatorname{sgn} \alpha \, e^{-c_3 |\alpha|}\, \d\alpha + c_2 c_3 \int_{-\Delta/2}^{\Delta/2} u(\alpha) \alpha \, e^{-c_3 |\alpha|} \, \d\alpha  = 0 ;\\
        c_1 u''(0) + (2 c_2 - c_1 c_3^2 ) u(0) - 2 c_2 c_3 \int_{-\Delta/2}^{\Delta/2} u(\alpha) e^{-c_3 |\alpha|} \, \d\alpha =  - c_3^2;\\
        c_1 u'''(0) - (c_1 c_3^2 - 2 c_2) u'(0) -  2 c_2 c_3^2 \int_{-\Delta/2}^{\Delta/2} u(\alpha) \operatorname{sgn} \alpha \, e^{-c_3 |\alpha|}\, \d\alpha = 0
    \end{cases}
\end{equation*}
in the interval $[-\frac{\Delta}{2},\frac{\Delta}{2}]$ and identically zero outside it. 

Knowing that the kernel $K_\nu(0,z)$ must be even allows us to conclude that $u$ must likewise be even, which further simplifies the boundary conditions to
\begin{equation}\label{eq:lastodeBC}
    \begin{cases}
        c_1 u(0) + c_2 \int_{-\Delta/2}^{\Delta/2} u(\alpha)|\alpha|e^{-c_3 |\alpha|} \, \d\alpha = 1 ;\\
        u'(0)  = 0;\\
        c_1 u''(0) + (2 c_2 - c_1 c_3^2 ) u(0) - 2 c_2 c_3 \int_{-\Delta/2}^{\Delta/2} u(\alpha) e^{-c_3 |\alpha|} \, \d\alpha =  - c_3^2;\\
         u'''(0) = 0.
    \end{cases}
\end{equation}

To emphasize what are the crucial parameters for the solution to the ODE \eqref{eq:odeforthepositivec3} we denote by $\lambda=c_2/c_1$. By rewriting the equation, we get that $u$ is a solution to the following differential equation:
\begin{equation*}
       u^{(4)}(\xi) + 2(\lambda -  c_3^2)u''(\xi) + (2 \lambda c_3^2 +  c_3^4) u(\xi) =   \frac{c_3^4}{c_1}.
\end{equation*}
Thus $u(\xi)=v(\xi)+\mu$, where $v(\xi)$ is a solution to the homogeneous ODE with constant coefficients. To obtain a representation of the function $v(\xi)$ we have to solve the corresponding characteristic equation \eqref{eq:characteristic}.
Since we are assuming $\lambda, c_3>0$, all solutions to \eqref{eq:characteristic} are different from $0$, and it follows that 
\begin{equation*}
    \eta^2=\begin{cases}
        c_3^2-\lambda\pm\sqrt{\lambda}\sqrt{\lambda-4c_3^2}, \quad \text{ if } \lambda\geq 4c_3^2;\\
        c_3^2-\lambda\pm i\sqrt{\lambda}\sqrt{4c_3^2-\lambda}, \quad \text{ if } \lambda< 4c_3^2
    \end{cases}
\end{equation*}
for the solutions $\eta$. Since the expression on the right-hand side may never be non-negative, we conclude that there are no real solutions to equation \eqref{eq:characteristic}. Suppose the numbers $\eta_1, \eta_2, -\eta_1, -\eta_2$ are all the solutions listed according to the multiplicity (i.e., $\eta_2$ may be equal to $\eta_1$). We assume $\operatorname{Re}\eta_1, \operatorname{Re}\eta_2\ge 0$. 
The following situations may occur:\\
\emph{Case 1:} $\eta_2\neq\eta_1$ $\Longleftrightarrow \lambda\neq 4c_3^2$.\\
\emph{Case 2:} $\eta_2 = \eta_1 \Longleftrightarrow \lambda=4c_3^2$.\\
Assuming $\eta_2\neq \eta_1$ we can write the general solution to the ODE \eqref{eq:odeforthepositivec3} by 
\begin{equation*}
    u(\xi)=\widetilde{T}_1 e^{\eta_1 \xi}+\widetilde{T}_2 e^{\eta_2 \xi}+\widetilde{T}_3 e^{-\eta_1 \xi}+\widetilde{T}_4 e^{-\eta_2 \xi}+\mu
\end{equation*}
for all $\xi \in [-\frac{\Delta}{2},\frac{\Delta}{2}]$. Otherwise, in the case $\eta_2=\eta_1$,  the solution can be written as 
\begin{equation*}
     u(\xi)=(\widetilde{T}_1 +\widetilde{T}_2 \xi) e^{\eta_1 \xi}+ (\widetilde{T}_3+\widetilde{T}_4 \xi)e^{-\eta_1 \xi}+\mu 
\end{equation*}
for all $\xi \in [-\frac{\Delta}{2},\frac{\Delta}{2}]$. Note that the function $K_\nu(0, \cdot)$ is a real entire\footnote{We call a function $f: \C \to \C$ \emph{real entire} if $f$ is analytic in $\C$ and $f(x)$ is real for all $x$ real.} and even function. Thus its Fourier transform, $u(\xi)$, is a real-valued,  even function.

Suppose we are in the case where $\lambda\neq 4c_3^2$. The above-mentioned facts about $u(\xi)$ lead us to the following reduction:
\begin{equation*}
    u(\xi)= T_1 \cosh{\eta_1 \xi} +T_2\cosh{\eta_2 \xi}+\mu
\end{equation*}
for all $\xi \in [-\frac{\Delta}{2},\frac{\Delta}{2}]$, where $T_1$ and $T_2$ are appropriate complex numbers. We will identify these coefficients using the conditions in \eqref{eq:lastodeBC}, which become
\begin{align}
    \frac{1}{c_1}&= u(0) + \lambda \int_{-\Delta/2}^{\Delta/2} u(\alpha) \lvert\alpha\rvert\ e^{-c_3 \lvert\alpha\rvert } \d\alpha,\label{eq:initialdataeq1}\\
    - \frac{c_3^2}{c_1}&= u''(0) + (2 \lambda -  c_3^2 ) u(0) - 2 \lambda c_3 \int_{-\Delta/2}^{\Delta/2} u(\alpha) e^{-c_3 |\alpha|} \, \d\alpha.\label{eq:initialdataeq2}
\end{align}
Using the notation from \eqref{eq:auxiliaryA} and \eqref{eq:auxiliaryB} in the equations \eqref{eq:initialdataeq1} and \eqref{eq:initialdataeq2} we get 
\begin{equation*}
    \left\{\begin{aligned}
         \frac{1}{c_1}-A(0)\mu& = A(\eta_1) T_1+  A(\eta_2)T_2  \\
        - \frac{c_3^2}{c_1} -B(0)\mu& = B(\eta_1)T_1 +B(\eta_2)T_2.
    \end{aligned}\right.
\end{equation*}
In order to have a unique solution to this system of equations, we need to check whether the quantity 
\begin{align*}
    A(\eta_1)&B(\eta_2)-B(\eta_1)A(\eta_2)
\end{align*} is non-zero or not. Assuming $\lambda \Delta^2\leq 5/3$ we can ensure that the quantity is not zero. The details for this can be found in the appendix.

It follows from these equations that
\begin{align*}
    T_1&=\frac{1}{A(\eta_1)B(\eta_2)-B(\eta_1)A(\eta_2)}\left\{\left(\frac{1}{c_1}-A(0)\mu\right)B(\eta_2)+\left(\frac{c_3^2}{c_1}+B(0)\mu\right)A(\eta_2)\right\}\\
\intertext{and}
    T_2&=-\frac{1}{A(\eta_1)B(\eta_2)-B(\eta_1)A(\eta_2)}\left\{\left(\frac{1}{c_1}-A(0)\mu\right)B(\eta_1)+\left(\frac{c_3^2}{c_1}+B(0)\mu\right)A(\eta_1)\right\}.
\end{align*}
Thus, 
\begin{equation*}
    \begin{split}
        u(\xi)=\frac{1}{A(\eta_1)B(\eta_2)-B(\eta_1)A(\eta_2)}\left[\left\{B(\eta_2)\left(\frac{1}{c_1}-A(0)\mu\right)+A(\eta_2)\left(\frac{c_3^2}{c_1}+B(0)\mu\right)\right\}\cosh{\eta_1\xi}\right.\\
        -\left.\left\{A(\eta_1)\left(\frac{c_3^2}{c_1}+B(0)\mu\right)+B(\eta_1)\left(\frac{1}{c_1}-A(0)\mu\right)\right\}\cosh{\eta_2\xi}\right]+\mu.
    \end{split}
\end{equation*}
In the second case, one can repeat this argument with a slight modification and get the following expression
\begin{equation*}
   \begin{split}
        u(\xi)=\frac{1}{A(\eta_1)B'(\eta_1)-A'(\eta_1)B(\eta_1)}\left[\left\{B'(\eta_1)\left(\frac{1}{c_1}-A(0)\mu\right)+A'(\eta_1)\left(\frac{c_3^2}{c_1}+B(0)\mu\right)\right\}\cosh{\eta_1\xi}\right.\\
        -\left.\left\{A(\eta_1)\left(\frac{c_3^2}{c_1}+B(0)\mu\right)+B(\eta_1)\left(\frac{1}{c_1}-A(0)\mu\right)\right\}\cdot \xi\cdot \sinh{\eta_1\xi}\right]+\mu,
    \end{split}
\end{equation*}
where $A'=\frac{d A}{d\eta}$ and $B'=\frac{d B}{d\eta}$.

By taking the inverse Fourier transform of $u$, we get the expression for $k_0$ of item (iii) in Lemma \ref{lem:functionalodeexp}, which is real entire, and so it also yields the formula for $K_\nu(0,z)$ in \eqref{eq:reprokernelkzero}.
\end{proof}

\section{Appendix}
Recall the definitions of $A(\eta)$ and $B(\eta)$ from \eqref{eq:auxiliaryA} and \eqref{eq:auxiliaryB}. We prove the following:
\begin{lemma}\label{lem:nonvanishing}
    For $0< \frac{c_2}{c_1}\Delta^2\leq 2.9$ and $0 < c_3^2 \neq \frac{c_2}{4 c_1}$, we have
\begin{align*}
   \mathcal{L}:= A(\eta_1) B(\eta_2)-B(\eta_1)A(\eta_2) \neq 0,
\end{align*}
where $\eta_1$ and $\eta_2$ are two roots of \eqref{eq:characteristic} such that $\eta_1+\eta_2\neq 0$.
\end{lemma}
\begin{proof}[Proof of Lemma \ref{lem:nonvanishing}]
As previously, $\lambda := c_2/c_1$. Let us set the notation $I_k(\eta)$ as
\begin{align*}
    I_k(\eta):&=\int_{-\Delta/2}^{\Delta/2} \cosh{(\eta \alpha)}\lvert\alpha\rvert^k e^{-c_3 \lvert\alpha\rvert }\,d\alpha, 
\end{align*} 
for $\eta \in \mathbb{C}$ and $k\in \{0, 1\}$. Thus,  \eqref{eq:auxiliaryA} and \eqref{eq:auxiliaryB} become 
\begin{align*}
A(\eta)=1+\lambda I_1(\eta) \quad\text{ and }\quad
B(\eta)=\eta^2+2\lambda-c_3^2-2\lambda c_3I_0(\eta).
\end{align*}
Define further
\begin{equation*}
    \Lambda(\lambda, c_3)=\begin{cases}
        \sqrt{\lambda-4c_3^2}, \quad \text{ if } \lambda\geq 4c_3^2;\\
        i\sqrt{4c_3^2-\lambda}, \quad \text{ if } \lambda< 4c_3^2.
    \end{cases}
\end{equation*}
From the assumption $\lambda \neq 4c_3^2$, we obtain, as in the proof of Theorem \ref{thm:reprokernelkzero}, that $\eta_1 \neq \eta_2$. From the hypothesis $\eta_1 + \eta_2 \neq 0$, we can assume without loss of generality that
\begin{align}\label{two solution}
    \eta_1^2 &= c_3^2 - \lambda + \sqrt{\lambda} \Lambda(\lambda, c_3) \quad\mbox{ and }\quad
    \eta_2^2 = c_3^2 - \lambda - \sqrt{\lambda} \Lambda(\lambda, c_3) .
\end{align}
Expanding and regrouping the terms of $ \mathcal{L}$ leads us to the following
\begin{equation}\label{eq:case1expression}
    \begin{split}
       \mathcal{L}=-2\sqrt{\lambda}\Lambda(\lambda, c_3)&+\lambda^2\big(I_1(\eta_1)-I_1(\eta_2)\big)-\lambda\sqrt{\lambda}\Lambda(\lambda, c_3)\big(I_1(\eta_1)+I_1(\eta_2)\big)\\
        &+2\lambda c_3\big(I_0(\eta_1)-I_0(\eta_2)\big)-2\lambda^2 c_3\big(I_0(\eta_2) I_1(\eta_1)-I_0(\eta_1) I_1(\eta_2)\big)\\
       = L_1+L_2,
    \end{split}
\end{equation}
where 
\begin{align*}
    &L_1:= -\sqrt{\lambda}\Lambda(\lambda, c_3)\big(2+\lambda\big(I_1(\eta_1)+I_1(\eta_2)\big)\big)\\
    \intertext{and}
    & L_2:= \lambda^2\big(I_1(\eta_1)-I_1(\eta_2)\big)+2\lambda c_3\big(I_0(\eta_1)-I_0(\eta_2)\big)-2\lambda^2 c_3\big(I_0(\eta_2) I_1(\eta_1)-I_0(\eta_1) I_1(\eta_2)\big).
\end{align*}

Let $\sigma := \lambda \Delta^2$, which by hypothesis satisfies $\sigma \leq 2.9$.
We can treat the problem in two disjoint cases:
(I)  $c_3<\frac{\sqrt{\lambda}}{2}$ and (II) $c_3>\frac{\sqrt{\lambda}}{2}$.

We proceed with case (I). Since the solutions of \eqref{eq:characteristic}  are purely imaginary we can take $\eta_1 = i y_1, \text{ and } \eta_2 = i y_2, \text{where } 0 < y_1 < y_2$. This choice of $\eta_1$ and $\eta_2$ implies no loss of generality because $A$ and $B$ are even functions and the non-vanishing of $\mathcal{L}$ is independent of this choice. 

For $k\in \{0,1\}$, by using a trigonometric identity for $\cos{x}-\cos{y}$, we can compute
\begin{align*}
    \left| I_k(\eta_1) \right.- & \left.I_k(\eta_2) \right|  = 4 \left| \int_0^{\Delta/2} \sin\left(\frac{y_1 + y_2}{2} \; \alpha \right)  \sin\left(\frac{y_1 - y_2}{2} \; \alpha \right) \alpha^k  e^{-c_3 \alpha} \d\alpha \right| \\
    &\leq 4 \frac{\left| y_1 - y_2 \right|}{2} \cdot\frac{y_1+y_2}{2}\int_0^{\Delta/2} \alpha^{2+k}\; \d\alpha=  (\eta_1^2 - \eta_2^2)\frac{\Delta^{k+3}}{(k+3)2^{k+3}}= \sqrt{\lambda}\sqrt{\lambda - 4c_3^2}\frac{\Delta^{k+3}}{(k+3)2^{k+2}}.
\end{align*}
On the other hand, by  employing the identity
\begin{equation*}
    \begin{split}
        \cos(y_2 \alpha)\cos(y_1 \beta) - \cos(y_1 \alpha) \cos(y_2 \beta)= &\sin\left(\frac{(y_1+y_2)(\alpha+\beta)}{2}\right)\sin\left(\frac{(y_1-y_2)(\alpha-\beta)}{2}\right)\\
        &\qquad\quad+\sin\left(\frac{(y_1+y_2)(\alpha-\beta)}{2}\right)\sin\left(\frac{(y_1-y_2)(\alpha+\beta)}{2}\right)
    \end{split}
\end{equation*}
and the inequality $\sin{x}\leq x$, we derive that
\begin{align*}
     \left\lvert I_0(\eta_2) I_1(\eta_1) \right. &- \left. I_0(\eta_1) I_1(\eta_2) \right\rvert \leq 2(y_2^2-y_1^2)\int\limits_0^{\Delta/2} \int\limits_0^{\Delta/2} \left\lvert\alpha^2-\beta^2\right\rvert\beta\; \d\alpha \d\beta 
     = \frac{11}{480}\Delta^5\sqrt{\lambda}\sqrt{\lambda - 4c_3^2}.
    \end{align*}
Combining the two upper bounds above we get
\begin{equation*}
       |L_2| \leq \sqrt{\lambda}\sqrt{\lambda-4c_3^2}\bigg(\frac{\lambda^{2}\Delta^4}{32}+\frac{\lambda^{3/2}\Delta^3}{12}+\frac{11\lambda^{2}\Delta^5\sqrt{\lambda}}{480}\bigg) = \sqrt{\lambda}\sqrt{\lambda-4c_3^2}\bigg(\frac{\sigma^2}{32}+\frac{\sigma^{3/2}}{12}+\frac{11\sigma^{5/2}}{480}\bigg).
\end{equation*}
Since we know $0<y_1<y_2 = \sqrt{\lambda - c_3^2 + \sqrt{\lambda}\sqrt{\lambda - 4 c_3^2}} <\sqrt{2\sigma}/{\Delta}$ and $\sigma\leq 2.9 < \frac{\pi^2}{2}$, this implies that $ y_1\alpha, y_2\alpha \in \left[0, \frac{\pi}{2}\right],$ for all $\alpha\in \left[0, \frac{\Delta}{2}\right]$. Therefore, $\cos(y_1\alpha)\geq\cos(y_2\alpha)\geq0$ for all $\alpha\in \left[0, \frac{\Delta}{2}\right]$ and hence  $I_1(\eta_1)+I_1(\eta_2)>0$. Thus, $L_1<-2\lambda \sqrt{\lambda}\sqrt{4c_3^2-\lambda}$ holds.

Finally, by our assumption $\sigma \leq 2.9$, we have 
\begin{align*}
    \mathcal{L} \leq L_1+|L_2| < \sqrt{\lambda}\sqrt{\lambda-4c_3^2}\bigg(\frac{\sigma^2}{32}+\frac{\sigma^{3/2}}{12}+\frac{11\sigma^{5/2}}{480} -2\bigg)<0.
\end{align*}
 
 Now, we treat case (II). Here, the solutions to \eqref{eq:characteristic} are of the form $\eta_1, -\eta_1, \overline{\eta_1}, -\overline{\eta_1}$. We make the choice that $\eta_1$ should be in the first quadrant of the plane, and select $\eta_2=\overline{\eta_1}$. Again, this selection implies no loss of generality because $A$ and $B$ are even functions and the non-vanishing of $\mathcal{L}$ is independent of this choice. Given this, we can see $\mathcal{L}$ is a purely imaginary number. 
 
 Let $\eta_1 =x_1+ i y_1$, where $x_1>0$ and  $y_1>0$, then $\eta_2 =x_1- i y_1$ and $\eta_1^2 = c_3^2 - \lambda + i\sqrt{\lambda}\sqrt{4c_3^2-\lambda}$ which gives
\begin{align*}
    x_1=c_3\sqrt{\frac{1}{2}-\frac{\lambda}{2c_3^2}+\frac{1}{2}\sqrt{1+\frac{2\lambda}{c_3^2}}} \quad\text{ and }\quad    y_1=\frac{\sqrt{\lambda}\sqrt{4c_3^2-\lambda}}{2c_3\sqrt{\frac{1}{2}-\frac{\lambda}{2c_3^2}+\frac{1}{2}\sqrt{1+\frac{2\lambda}{c_3^2}}}}.
\end{align*}
Notice that $\sigma\leq 2.9 < \pi^2$, so we can identify the signs of the terms involving trigonometric factors in $I_1(\eta_1)$ and $I_1(\eta_2)$, because $0<y_1<\sqrt{\lambda}$, and we have $ y_1\alpha \in [0, \pi/2]$ for $0<\alpha<\Delta/2$. Hence, $I_1(\eta_1)+I_1(\eta_2) >0$ and, consequently,
 \begin{align}\label{ImL_1}
           \operatorname{Im}(L_1)<-2\lambda \sqrt{\lambda}\sqrt{4c_3^2-\lambda}.
        \end{align}
For $k\in \{0,1\}$, we have
\begin{align*}
        \operatorname{Im}\bigg\{I_k(\eta_1)-I_k(\eta_2)\bigg\}& =2\operatorname{Im}\left\{I_k(x_1+i y_1)\right\}=4\int_0^{\Delta/2}\sinh(x_1\alpha)\sin(y_1\alpha)\alpha^k e^{-c_3\alpha}\,\d\alpha.
\end{align*}
Also, by using a trigonometric identity for $\cosh(x+y)$ and a few simplifications, we get
\begin{align*}
    \operatorname{Im}\bigg\{ & I_0(\eta_2)I_1(\eta_1) -  I_0(\eta_1)I_1(\eta_2)\bigg\} =2\operatorname{Im}\left\{I_0(x_1-i y_1)I_1(x_1+i y_1)\right\}\\
    &=4\int\limits_{0}^{\Delta/2}\int\limits_{0}^{\Delta/2}[\sinh(x_1(\alpha+\beta))\sin(y_1(\alpha-\beta))+ \sinh(x_1(\alpha-\beta))\sin(y_1(\alpha+\beta))] \alpha e^{-c_3(\alpha+\beta)}\,\d\alpha\,\d\beta.
    \end{align*}
    
At this point, we have to consider two sub-cases: (a) $c_3\in \left(\frac{\sqrt{\lambda}}{2}; \sqrt{\lambda}\right]$ and (b) $c_3>\sqrt{\lambda}$.

For subcase (a), notice that we always have $0<x_1<c_3$ and so, for $t \in \R$,
\begin{equation*}
    |\sinh(x_1 t)|=\frac{1}{2}\left\lvert e^{x_1 t}-e^{-x_1 t} \right\rvert= \left|x_1 t\right| e^{\xi t}, 
\end{equation*}
for some  $\xi \in [-x_1, x_1]\subseteq[-c_3, c_3]$. Therefore, we can say
\begin{equation*}
    \left\lvert \sinh(x_1(\alpha+\beta))\sin(y_1(\alpha-\beta)) \right\rvert\leq x_1 y_1\lvert \alpha^2-\beta^2\rvert \cdot e^{c_3(\alpha+\beta)}= \frac{1}{2}\sqrt{\lambda}\sqrt{4c_3^2-\lambda}\,\lvert \alpha^2-\beta^2\rvert \cdot e^{c_3(\alpha+\beta)}
\end{equation*}
for any $\alpha, \beta >0$.
Using this estimate, we deduce that 
 \begin{align*}
        \left\lvert \operatorname{Im}\bigg\{I_0(\eta_2)I_1(\eta_1) -  I_0(\eta_1)I_1(\eta_2)\bigg\}\right\rvert&\leq 4\sqrt{\lambda}\sqrt{4c_3^2-\lambda}\int\limits_{0}^{\Delta/2}\int\limits_{0}^{\Delta/2}\lvert \alpha^2-\beta^2\rvert \alpha \,\d\alpha\,\d\beta = \frac{11\Delta^5}{480}\sqrt{\lambda}\sqrt{4c_3^2-\lambda}.
    \end{align*}
Also, for $k\in \{0,1\}$ we obtain using $x_1 y_1 = \frac{1}{2}\sqrt{\lambda}\sqrt{4c_3^2-\lambda}$ that
\begin{align*}
        \left\lvert \operatorname{Im}\bigg\{I_k(\eta_1)-I_k(\eta_2)\bigg\} \right \rvert &\leq 4 x_1 y_1\int_0^{\Delta/2} \alpha^{2+k}\,\d\alpha = \frac{\Delta^{3+k}}{(3+k)2^{k+2}}\sqrt{\lambda}\sqrt{4c_3^2-\lambda}.
        \end{align*}
        Hence, we get
    \begin{align*}
      |\operatorname{Im}(L_2)|\leq  \sqrt{\lambda}\sqrt{4c_3^2-\lambda}\left(\frac{11\lambda^{5/2}\Delta^5}{240}+\frac{\lambda^2\Delta^4}{32}+\frac{\lambda^{3/2}\Delta^3}{6}\right)
        = \sqrt{\lambda}\sqrt{4c_3^2-\lambda}\left(\frac{11}{240}\sigma^{5/2}+\frac{1}{32}\sigma^2+\frac{1}{6}\sigma^{3/2}\right) .   
    \end{align*}
Whenever $\sigma \leq 3$, combining the bound of $|\operatorname{Im}(L_2)|$ with \eqref{ImL_1} we get
\begin{align*}
        \operatorname{Im}\big\{\mathcal{L}\big\} \leq \operatorname{Im}(L_1) + |\operatorname{Im}(L_2)| < \sqrt{\lambda}\sqrt{4c_3^2-\lambda}\left(\frac{11}{240}\sigma^{5/2}+\frac{1}{32}\sigma^2+\frac{1}{6}\sigma^{3/2}-2\right) <0.
\end{align*}

In subcase (b), we have $c_3>\sqrt{\lambda}$ and 
\begin{equation*}
    |\sinh(x_1 (\alpha\pm\beta))e^{-c_3(\alpha+\beta)}|\leq \frac{1}{2},
\end{equation*}
for any $\alpha, \beta>0$. Using this estimate, we deduce that
\begin{align*}
        \left\lvert \operatorname{Im}\bigg\{I_0(\eta_2)I_1(\eta_1)\right. - \left. I_0(\eta_1)I_1(\eta_2)\bigg\}\right\rvert&\leq 2y_1\int\limits_{0}^{\Delta/2}\int\limits_{0}^{\Delta/2}\bigg(\lvert \alpha-\beta\rvert+|\alpha+\beta|\bigg) \alpha \,d\alpha\,d\beta =\frac{3\Delta^4 y_1}{32},
    \end{align*}  
and since $y_1 < \sqrt{\lambda}$, for $k \in \{0,1\}$,
        \begin{align*}
        \left\lvert \operatorname{Im}\bigg\{I_k(\eta_1)-I_k(\eta_2)\bigg\} \right \rvert &\leq 2y_1\int_0^{\Delta/2} \alpha^{k+1}\,d\alpha=\frac{y_1\Delta^{k+2}}{(k+2)2^{k+1}}\leq \frac{\sqrt{\lambda}\Delta^{k+2}}{(k+2)2^{k+1}}.
        \end{align*}
The two upper bounds above give 
\begin{align*}
        |\operatorname{Im}(L_2)| \leq c_3\sqrt{\lambda}\left(\frac{3}{16}\lambda^2{\Delta^4}+\frac{\lambda^{3/2}\Delta^3}{12}+\frac{\lambda\Delta^2}{2}\right) = c_3\sqrt{\lambda}\left(\frac{3}{16}\sigma^{2}+\frac{1}{12}\sigma^{3/2}+\frac{1}{2}\sigma\right).
    \end{align*}
Since $\lambda/c_3^2<1$, combining the upper bound of $\operatorname{Im}(L_2)$ with \eqref{ImL_1} we obtain that 
\begin{equation*}
        \operatorname{Im}\big\{\mathcal{L}\big\} \leq \operatorname{Im}(L_1) + |\operatorname{Im}(L_2)| <  c_3\sqrt{\lambda}\left(\frac{3}{16}\sigma^{2}+\frac{1}{12}\sigma^{3/2}+\frac{1}{2}\sigma -2 \sqrt{3}\right) <0,
\end{equation*}
whenever $\sigma\leq 2.9$.
\end{proof}

\section*{Acknowledgements}
The authors sincerely thank Emanuel Carneiro for his suggestions, discussions, and encouragement throughout this project. They are grateful to Steven Gonek for having read part of this manuscript and for his remarks about Conjecture 1.  They also thank Micah B. Milinovich for insightful discussions and Sohaib Khalid for helping us think outside of the box.


\begin{thebibliography}{99}

\bibitem{A}
N. I. Achieser,
\newblock{ Theory of Approximation},
\newblock Frederick Ungar Publishing, New York, 1956.

\bibitem{BMMRTW2015}
W. Barrett, B. McDonald, S. J. Miller, P. Ryan, C. L. Turnage-Butterbaugh, K. Winsor,
\newblock Gaps between zeros of $GL(2)$ $L$-functions,
\newblock J. Math. Anal. Appl. 429 (2015), no. 1, 204--232.


\bibitem{CCChiM}   
E.~Carneiro, V.~Chandee, A. Chirre, and M.~B.~Milinovich,
\newblock On Montgomery's pair correlation conjecture: a tale of three integrals, 
\newblock J.~Reine Angew.~Math.~786 (2022), 205--243.

\bibitem{CCLM}
E.~Carneiro, V.~Chandee, F.~Littmann, and M.~B.~Milinovich, 
\newblock Hilbert spaces and the pair correlation of zeros of the Riemann zeta-function,
\newblock J.~Reine Angew.~Math.~725 (2017), 143--182.


\bibitem{CMR}
E.~Carneiro, M.~B.~Milinovich, A.~P.~Ramos,
\newblock Fourier optimization and Montgomery’s pair correlation conjecture,
\newblock Math. Comp.~94 (2025), 409--424.


\bibitem{CKL}
V.~Chandee, K.~Klinger-Logan, and X.~Li, 
\newblock Pair correlation of zeros of $\Gamma_1(q)$ $L$-functions,
\newblock  Math.~Z.~302 (2022), no.~1, 219--258.




\bibitem{CE}
H. Cohn, N. Elkies,
\newblock New upper bounds on sphere packings I,
\newblock Ann. of Math. (2) 157 (2003), no. 2, 689--714.


\bibitem{DRTW}
D.~de Laat, L. Rolen, Z. Tripp, I. Wagner,
\newblock Pair correlation for Dedekind zeta functions of abelian extensions, preprint available at \href{https://arxiv.org/abs/1908.04876}{https://arxiv.org/abs/1908.04876.}
\newblock 


\bibitem{FGL2014}
D. W. Farmer, S. M. Gonek, Y. Lee, 
\newblock Pair correlation of the zeros of the derivative of the Riemann $\xi$-function,
\newblock  J. Lond. Math. Soc. (2) 90 (2014), no. 1, 241--269.

\bibitem{Ga}
M. Z. Garaev,
\newblock On vertical zeros of {$\Re\zeta(s)$} and {$\Im\zeta(s)$},
\newblock Acta Arith. 108 (2003), no.3, 245--251.

\bibitem{G2} 
D.~A.~Goldston,
\newblock On the function $S(T)$ in the theory of the Riemann zeta-function,
\newblock J.~Number Theory 27 (1987), no.~2, 149--177.

\bibitem{G} 
D.~A.~Goldston,
\newblock On the pair correlation conjecture for zeros of the Riemann zeta-function,
\newblock J.~Reine Angew.~Math.~385 (1988), 24--40.

\bibitem{GG} 
D.~A.~Goldston and S.~M.~Gonek,
\newblock A note on the number of primes in short intervals,
\newblock Proc.~Amer.~Math.~Soc.~108 (1990), no.~3, 613--620.

\bibitem{GM1984}
D. A. Goldston, H. L. Montgomery,
\newblock Pair correlation of zeros and primes in short intervals,
\newblock Analytic number theory and Diophantine problems (Stillwater, OK, 1984), 183--203, Progr. Math., 70, Birkhäuser Boston, Inc., Boston, MA, 1987. 

\bibitem{GK2018}
S. M. Gonek, H. Ki,
\newblock Pair correlation of zeros of the real and imaginary parts of the Riemann zeta-function,
\newblock J. Number Theory 186 (2018), 35--61.


\bibitem{Ki}
H. Ki,
\newblock On the zeros of sums of the Riemann zeta function,
\newblock J. Number Theory 128(2008), no. 9, 2704--2755.

\bibitem{KY2017}
Y. Karabulut, C. Y. Yıldırım,
\newblock Some analogues of pair correlation of zeta zeros,
\newblock Exploring the Riemann zeta function, 113--179, Springer, Cham, 2017.


\bibitem{M1973}
 H.~L.~Montgomery, 
\newblock The pair correlation of zeros of the zeta function,
\newblock Analytic number theory (Proc. Sympos. Pure Math., Vol. XXIV, St.~Louis Univ., St.~Louis, Mo., 1972), pp. 181--193. Amer.~Math.~Soc., Providence, R.I., 1973.


\bibitem{M2}
H.~L.~Montgomery, 
\newblock Distribution of the zeros of the Riemann zeta function,
\newblock Proceedings of the International Congress of Mathematicians (Vancouver, B. C., 1974), Vol.~1, pp. 379--381. Canad.~Math.~Congress, Montreal, Que., 1975. 



\bibitem{MP1999}
M. R. Murty, A. Perelli, 
\newblock The pair correlation of zeros of functions in the Selberg class,
\newblock  Internat. Math. Res. Notices 1999, no. 10, 531--545. 




\bibitem{S1989}
A. Selberg,
\newblock Old and new conjectures and results about a class of Dirichlet series,
\newblock Proceedings of the Amalfi Conference on Analytic Number Theory (Maiori, 1989), Salerno: Univ. Salerno, pp. 367–385, Reprinted in Collected Papers, vol 2, Springer--Verlag, Berlin (1991)


 
\end{thebibliography}
\end{document}